\documentclass[a4paper,12pt]{article}
\usepackage[top=2.5cm,bottom=2.5cm,left=2.5cm,right=2.5cm]{geometry}
\usepackage{cite, amsmath, amssymb}
\numberwithin{equation}{section}
\usepackage{amssymb}
\usepackage{graphicx}
\newenvironment{proof}{{\noindent\textbf{\textit {Proof.}}}}{\hfill $\blacksquare$\par}
\usepackage{latexsym}
\usepackage{graphicx,booktabs,multirow}
\usepackage{tikz}
\usetikzlibrary{decorations.pathreplacing}
\usetikzlibrary{intersections}
\usepackage{enumerate}
\usepackage{graphicx,booktabs,multirow}
\usepackage{appendix}
\newtheorem{theorem}{Theorem}[section]
\newtheorem{proposition}[theorem]{\rm\bfseries Proposition}
\newtheorem{lemma}[theorem]{Lemma}
\newtheorem{corollary}[theorem]{\rm\bfseries Corollary}

\newtheorem{conjecture}[theorem]{Conjecture}
\newtheorem{problem}[theorem]{Problem}

\newtheorem{definition}[theorem]{Definition}
\usepackage[section]{placeins}
\def\NAT@def@citea{\def\@citea{\NAT@separator}}% Suppress spaces between citations using natbib.sty
\allowdisplaybreaks

\begin{document}
\vspace*{10mm}

\noindent
{\Large \bf The existence of a spanning tree with leaf distance at least $d$ and leaf degree at most $k$ via the size or the spectral radius with respect to the minimum degree}

\vspace*{7mm}

\noindent
{\large \bf  Jifu Lin, Lihua You$^*$}
\noindent

\vspace{7mm}

\noindent
School of Mathematical Sciences, South China Normal University,  Guangzhou, 510631, P. R. China,
e-mail: {\tt 2023021893@m.scnu.edu.cn}, {\tt ylhua@scnu.edu.cn},\\[2mm]
$^*$ Corresponding author
\noindent

%\footnotesize $^1${\it School of Mathematical Sciences, South China Normal University, Guangzhou, 510631, P. R. China}\\
%\footnotesize $^2${\it Department of Mathematics Teaching, Guangzhou Civil Aviation College, Guangzhou, 510403, P. R. China}\\
%\noindent
% $^2${\it Department of Mathematics Teaching, Guangzhou Civil Aviation College, Guangzhou, 510403, P. R. China\/} \\
\vspace{7mm}

\noindent
{\bf Abstract} \
\noindent
Let $k$, $d$ be a positive integer, $G$ be a connected graph of order $n$, $T$ be a tree. The leaf distance of a tree is defined as the minimum distance between any two leaves. For $v\in V(T)$, the leaf degree of $v$ in $T$ is the number of leaves adjacent to $v$, and the leaf degree of $T$ is defined as maximum leaf degree among the vertices of $T$. In this paper, motivated by the conjecture proposed by Kaneko (2001) and its subsequent partial confirmation by Erbes, Molla, Mousley and Santana (2017), we obtain lower bounds in terms of the size and the adjacent spectral radius to guarantee that $G$ contains a spanning tree with leaf distance at least $d$, where $4\leq d \leq n-1$. Furthermore, we obtain some tight conditions in $G$ for its size and spectral radius to ensure that $G$ has a spanning tree with leaf degree at most $k$, which improves the result of Ao, Liu, Yuan, Ng and Cheng (2023).
 \\[2mm]
%\vspace{5mm}


\baselineskip=0.30in

\section{Introduction}

\hspace{1.5em}Throughout this paper, we only consider simple, undirected and connected graphs. Let $G=(V,E)$ be a graph with vertex set $V(G)=\{v_1,v_2,\cdots,v_n\}$ and edge set $E(G)$, where $n$ and $|E(G)|$ are called the \emph{order} and the \emph{size} of $G$, respectively. For $S\subseteq V(G)$, let $N_G(S)=\bigcup_{v\in S}N_G(v)$, where $N_G(v)=\{u\in V(G):uv\in E(G)\}$. For $v\in V(G)$, the degree of $v$ in $G$, denoted by $d_G(v)$, is the number of vertices adjacent to $v$ in $G$. The minimum degree of $G$ is denoted by $\delta(G)$. A \emph{leaf} is a vertex with degree one in a tree. For a subset $S \subseteq V(G)$, we denote by $G[S]$ the subgraph of $G$ induced by $S$, and by $G-S$ the subgraph obtained from $G$ by removing the vertices in $S$ and their incident edges.

Let $G_1$ and $G_2$ be vertex-disjoint graphs. The \emph{union} $G_1 \cup G_2$ is a graph  with vertex set $V(G_1)\cup V(G_2)$ and edge set $E(G_1)\cup E(G_2)$. For any integer $t$, $tG$ denotes the union of $t$ disjoint copies of graph $G$. The \emph{join} $G_1 \vee G_2$ is formed by adding an edge between every vertex of $G_1$ and every vertex of $G_2$. For undefined terms, see \cite{JAB}.

Let $G$ be a graph with $n$ vertices. The \emph{adjacency matrix} $A(G)=(a_{ij})$ of $G$ is an $n\times n$ symmetric matrix, where $a_{ij}=1$ if vertices $v_i$ and $v_j$ are adjacent in $G$, and $a_{ij}=0$ otherwise. Let $D(G)$ be the diagonal matrix of vertex degrees in $G$. The matrix $Q(G)=D(G)+A(G)$ is known as the \emph{signless Laplacian matrix} of $G$. The largest eigenvalue of $A(G)$ (resp. $Q(G)$), denoted by $\rho(G)$ (resp. $q(G)$), is referred to as the adjacency spectral radius (resp. signless Laplacian spectral radius) of $G$.

A \emph{spanning tree} $T$ of a connected graph $G$ is a spanning subgraph of $G$ that is a tree, including all vertices of $G$. For an integer $k\geq 2$, a \emph{$k$-tree} is a tree with the maximum degree at most $k$, and a \emph{$k$-ended tree} is a tree with at most $k$ leaves. 

Let $T$ be a tree. For a vertex $v\in V(T)$, the \emph{leaf degree} of $v$ is defined as the number of leaves adjacent to $v$, and the leaf degree of $T$ is the maximum leaf degree among all vertices of $T$. The \emph{leaf distance} of $T$ is defined as the minimum distance between any two leaves in $T$.

%In graph theory, deciding whether a graph has spanning trees of particular types is a fundamental problem. For example, the Hamilton path problem, that is to find a spanning 2-ended tree, is a famous problem. 

In mathematical literature, the study on the existence of spanning trees attracted much attention. For instance, Win \cite{S1} established a link between the presence of spanning $k$-trees in a graph and its toughness, and proposed a Chvátal-Erdős type criterion for ensuring that a $t$-connected graph contains a spanning $k$-ended tree \cite{S2}. Gu and Liu \cite{XA} identified connected graphs with spanning $k$-trees by utilizing Laplacian eigenvalues. Fan, Goryainov, Huang, and Lin \cite{DT} investigated the existence of a spanning $k$-tree in a connected graph through the adjacent spectral radius and the signless Laplacian spectral radius. Kaneko \cite{AK1} introduced the concept of leaf degree in a spanning tree and provided a criterion for a connected graph to have a spanning tree with leaf degree at most $k$. Ao, Liu, and Yuan \cite{GA} derived the tight spectral conditions for the existence of a spanning $k$-ended tree and a spanning tree with leaf degree at most $k$ in a connected graph. Zhou, Sun, and Liu \cite{DQ} determined the $\mathcal{D}$-index and $\mathcal{Q}$-index for spanning trees with leaf degree at most $k$ in connected graphs, and other similar studies have been conducted.

The present paper is organized as follows. In Section 2, we recall some well - known results and prove several lemmas. In Sections 3 and 4, we discuss the existence of spanning trees with leaf distance at least $d$ and leaf degree at most $k$ in graphs, respectively, with respect to the minimum degree.

\section{Preliminaries}\label{sec-pre}

\hspace{1.5em}In this section, we introduce some necessary preliminary lemmas, which are very useful in the following proofs.

\begin{definition}{\rm(\!\!\cite{AE})}\label{d1}
	Let $M$ be a complex matrix of order $n$ described in the following
	block form\begin{equation*}
		M=\begin{pmatrix}
			M_{11}&\cdots&M_{1l}\\
			\vdots&\ddots&\vdots\\
			M_{l1}&\cdots&M_{ll}
		\end{pmatrix}
	\end{equation*}
	where the blocks $M_{ij}$ are $n_i \times n_j$ matrices for any $1\leq i, j \leq l$ and $n= n_1 +\cdots + n_l$.
	For $1 \leq i, j \leq l$, let $b_{ij}$ denote the average row sum of $M_{ij}$, i.e. $b_{ij}$ is the sum of all
	entries in $M_{ij}$ divided by the number of rows. Then $B(M)=(q_{ij})$ (or simply $B$) is
	called the quotient matrix of $M$. If, in addition, for each pair $i, j$, $M_{ij}$ has a constant row
	sum, i.e., $M_{ij}\vec{e}_{n_j}= b_{ij}\vec{e}_{n_i}$, then $B$ is called the equitable quotient matrix of $M$, where
	$\vec{e}_k=(1, 1,\cdots, 1)^T \in \mathcal{C}^k$, and $\mathcal{C}$ denotes the field of complex numbers.
\end{definition}

\begin{lemma}{\rm(\!\!\cite{ylh})}\label{Spectra}
	Let $B$ be the equitable quotient matrix of $M$ as defined in Definition \ref{d1}, and $M$ be a nonnegative matrix. Then $\lambda(B)=\lambda(M),$ where $\lambda(B)$ and $\lambda(M)$ are the spectral radius of $B$ and $M$, respectively.
\end{lemma}

\begin{lemma}{\rm(\!\!\cite{ylh})}\label{rq1}
	Let $G$ be a graph and $F$ be a spanning subgraph of $G$. Then $\rho(F)\leq \rho(G)$ and $q(F)\leq q(G)$. In particular, if $G$ is connected, and $F$ is a proper subgraph of $G$, then $\rho(F)<\rho(G)$ and $q(F)<q(G)$.
\end{lemma}

\section{The existence of spanning trees with leaf distance at least $d$ with respect to the minimum degree}	

\hspace{1.5em}Kaneko in \cite{AK1} posed the following conjecture on the existence of a spanning tree with larger leaf distance.

\begin{conjecture}{\rm(\!\!\cite{AK1})}\label{c2}
	Let $d\geq 3$ be an integer and $G$ be a connected graph of order $n\geq d+1$. If $i(G-S)<\frac{2|S|}{d-2}$ for any nonempty subset $S\subseteq V(G)$, then $G$ has a spanning tree with leaf distance at least $d$, where $i(G-S)$ is the number of isolated vertices in $G-S$.
\end{conjecture}

It is clear that if the leaf degree of a tree of order $n (\geq 3)$ is one, then its leaf distance is at least three, which implies \cite[Theorem 1]{AK1} proved Conjecture \ref{c2} for $d=3$. The following case $d=4$ was later proven by Kaneko, Kano, and Suzuki in \cite{AK2}. 

\begin{theorem}{\rm(\!\!\cite{AK2})}\label{t2}
	Let $G$ be a connected graph of order $n\geq 5$. If $i(G-S)<|S|$ for any nonempty subset $S \subseteq V(G)$, then $G$ has a spanning tree with leaf distance at least 4.
\end{theorem}

Given a graph $G$, let $\alpha(G)$ denote the independence number of $G$. For $k\leq \alpha(G)$, let $\delta_k(G)$ be the minimum order of the neighborhood of an independent set of order $k$ in a graph $G$. That is, $\delta_k(G)=\min\{|N_G(A)|\mid A$ is an independent set of order $k$ in $G$\}. Erbes, Molla, Mousley, and Santana \cite{CTS} provided the following proposition to translate the condition in Conjecture \ref{c2} into the language of neighborhood unions.

\begin{proposition}{\rm(\!\!\cite{CTS})}\label{p1}
	Let $d$ be an integer such that $d\geq 3$, and let $G$ be a connected graph. Then $i(G-S)<\frac{2|S|}{d-2}$ for all nonempty $S\subseteq V(G)$ if and only if $\delta_k(G)>\frac{k(d-2)}{2}$ for all $k$ satisfying $1\leq k\leq \alpha(G)$.
\end{proposition}

Thus Conjecture \ref{c2} can be rewritten in the following way.

\noindent\textbf{Conjecture \ref{c2}} (Equivalent Form).\emph{ Let $d\geq 3$ be an integer and $G$ be a connected graph of order $n\geq d+1$. If $\delta_k(G)>\frac{k(d-2)}{2}\text{ for all k satisfying } 1\leq k\leq \alpha(G), $ then $G$ has a spanning tree with leaf distance at least d.}	

Erbes, Molla, Mousley, and Santana \cite{CTS} proved the following theorem which implies Corollary \ref{c3}. For large values of $d$, Corollary \ref{c3} is a stronger version of Conjecture \ref{c2}.

\begin{theorem}{\rm(\!\!\cite{CTS})}\label{t5}
	Let $G$ be a connected graph of order $n$ and $d$ be an integer such that $n>d\geq 4$ and $\alpha(G)\leq 5$. If $\delta_{2k}(G)>k(d-2) \text{ for all k satisfying }1\leq k\leq \frac{\alpha(G)}{2},$ then $G$ has a spanning tree with leaf distance at least $d$.
\end{theorem}

By Proposition \ref{p1} and Theorem \ref{t5}, Conjecture \ref{c2} is true when $\alpha(G)\leq5$, which can be stated as follows. 

\begin{proposition}{\rm(\!\!\cite{CTS})}\label{p3}
	Let $d\geq 4$ be an integer and $G$ be a connected graph of order $n\geq d+1$ with $\alpha(G)\leq 5$. If $i(G-S)<\frac{2|S|}{d-2}$ for any nonempty subset $S\subseteq V(G)$, then $G$ has a spanning tree with leaf distance at least $d$.	
\end{proposition}

By Theorem \ref{t5}, Corollary \ref{c3} was obtained immediately, which showed Conjecture \ref{c2} is true when $d\geq \frac n3$.

\begin{corollary}{\rm(\!\!\cite{CTS})}\label{c3}
	Let $G$ be a connected graph of order $n$ and $d$ be an integer such that $n>d\geq \frac n3$. If $\delta_{2k}(G)>k(d-2)$ for all $k$ satisfying $1\leq k\leq \frac{\alpha(G)}{2}$, then $G$ has a spanning tree with leaf distance at least $d$.
\end{corollary}

Lin, You \cite{JL} and Chen, Lv, Li, Xu \cite{LD} studied some sufficient conditions to ensure that a graph $G$ with minimum degree $\delta(G)$ has a spanning tree with leaf distance at least four by using Theorem \ref{t2}, and they also in \cite{JL} showed the following problem.

\begin{problem}{\rm(\!\!\cite{JL})}\label{p2}
	For a connected graph $G$ of order $n$ and $5\leq d\leq n-1$, find some spectral radius conditions of $G$ to determine whether $G$ has a spanning tree with leaf distance at least $d$. 
\end{problem}

Motivated by \cite{JL}, in this paper, we study the sufficient conditions to ensure that a graph $G$ with the minimum degree $\delta(G)$ has a spanning tree with leaf distance at least $d$ (where $4\leq d\leq n-1$), and obtain the following theorems by using Theorem \ref{t5} and Corollary \ref{c3}. For any positive integer $m$ and $d\geq 4$, we define a function $\phi(m)=\lceil{\frac{2m}{d-2}}\rceil$.

\begin{theorem}\label{tt1}
	Let $G$ be a connected graph with order $n>d\geq 4$, $\delta(G)\geq t$, $d-2\mid 2t$ and $\alpha(G)\leq 5$, where $t$ and $d$ are positive integers. If $$|E(G)|>\begin{cases}
		\binom{n-\phi(t)}{2}+t\phi(t)+4t-3,& \text{if  $n=6$,}\\
		\binom{n-\phi(t)}{2}+t\phi(t)+4t-4,& \text{if  $n=5$ or $n\geq7$,}
	\end{cases}
	$$then $G$ has a spanning tree with leaf distance at least $d$.
\end{theorem}

\begin{theorem}\label{tt2}
	Let $G$ be a connected graph with order $n\geq 3d\geq 12$ {\rm{(i.e. $4\leq d\leq \frac{n}{3}$)}},  $\delta(G)\geq t$ and $\alpha(G)\leq 5$, where $t$ and $d$ are positive integers. If $$|E(G)|>\begin{cases}
		\binom{n-2}{2}+2(d-2),& \text{if $t\leq d-2$,}\\
		\binom{n-4}{2}+8(d-2),& \text{if $t>d-2$,}
	\end{cases}
	$$ then $G$ has a spanning tree with leaf distance at least $d$. 
\end{theorem}

\begin{theorem}\label{tt3}
	Let $G$ be a connected graph of order $n\geq d+1\geq6$. If $$|E(G)|>\begin{cases}
		\binom{n-6}{2}+6(n-6),& \text{if $\frac n3\leq d\leq \frac{n-3}{2}$,}\\
		\binom{n-4}{2}+8(d-2),& \text{if $\frac{n-2}{2}\leq d\leq \frac{n-1}{2},$}\\
		\binom{n-4}{2}+4(n-4),& \text{if $\frac{n}{2}\leq d\leq n-3$,}\\
		\binom{n-2}{2}+2(d-2),& \text{if $n-2\leq d\leq n-1$,}
	\end{cases}
	$$ then $G$ has a spanning tree with leaf distance at least $d$.
\end{theorem}

\begin{theorem}\label{tt4}
	Let $G$ be a connected graph with order $n\geq 6t+5$ and $\alpha(G)\leq5$, $t$, $d$ be positive integers with $t\leq \delta(G)$, $4\leq d\leq n-1$ and $d-2\mid 2t$. If $\rho(G)\geq\rho(K_t\vee (K_{n-t-\phi(t)}\cup\phi(t)K_1))$, then $G$ has a spanning tree with leaf distance at least $d$.
\end{theorem}

\begin{theorem}\label{tt5}
	Let $G$ be a connected graph with order $n\geq3d\geq 12$ {\rm{(i.e. $4\leq d\leq \frac{n}{3}$)}}, $\delta(G)\geq t$ and $\alpha(G)\leq 5$, where $t$ and $d$ are positive integers. If $$\rho(G)\geq\begin{cases}
		\rho(K_{d-2}\vee(K_{n-d}\cup2K_1)),& \text{if $t\leq d-2$,}\\
		\rho(K_{2(d-2)}\vee(K_{n-2d}\cup4K_1)),& \text{if $t>d-2$,}
	\end{cases}
	$$then $G$ has a spanning tree with leaf distance at least $d$. 
\end{theorem}

\begin{theorem}\label{tt6}
	Let $G$ be a connected graph of order $n\geq 15$. If \begin{equation}\label{th 1.13}
		\rho(G)\geq\begin{cases}
			\rho(K_{2(d-2)}\vee(K_{n-2d}\cup4K_1)),& \text{if $\frac{n-2}{2}\leq d\leq \frac{n-1}{2},$}\\
			\rho(K_{n-4}\vee4K_1),& \text{if $\frac{n}{2}\leq d\leq n-3$,}\\
			\rho(K_{d-2}\vee(K_{n-d}\cup2K_1)),& \text{if $n-2\leq d\leq n-1$,}
		\end{cases}
	\end{equation}
	then $G$ has a spanning tree with leaf distance at least $d$. 
\end{theorem}

The proofs of Theorems \ref{tt1} - \ref{tt6} will be provided as follows. Firstly, we provide the following lemmas.

\begin{lemma}\label {rho}
	Let $n$, $l$, $d$ be positive integers with $n>ld$, $H_l=K_{l(d-2)}\vee(K_{n-ld}\cup2lK_1)$, $f_l(x)=x^{3}-\left(n-2 l-2\right) x^{2}-\left(2 d l^{2}-4 l^{2}-2 l+n-1\right) x-2 d^{2} l^{3}+4 d l^{3}+2 d l^{2} n-2 d l^{2}-4 l^{2} n+4 l^{2}.$ Then $\rho(H_l)$ is the largest root of $f_l(x)=0$. In particular, we have 
	
	\noindent{\rm (i)} $n-3<\rho(H_1)<n+\frac{2d}{n}-3$, where $4\leq d\leq n-1$.
	
	\noindent{\rm (ii)} $\rho(H_2)>n+\frac dn-3$, where $n\geq15$ and $\frac{n-2}{2}\leq d\leq n-1$.	
\end{lemma}

\begin{proof}
	We consider the equitable partition $V(H_l)=V(K_{l(d-2)})\cup V(K_{n-ld})\cup V(2lK_1)$. The corresponding quotient matrix of $A(H_l)$ can be written as $$B_l=\begin{pmatrix}
		l(d-2)-1& n-ld& 2l\\
		l(d-2)& n-ld-1& 0\\
		l(d-2)& 0& 0\\
	\end{pmatrix},$$ and the characteristic polynomial of the matrix $B_l$ is $f_l(x)$. In light of Lemma \ref{Spectra}, $\rho_l=\rho(H_l)$ is the largest root of $f_l(x)=0$. 
	
	By direct computation, we obtain $f_1(n-3)=-2 \left(d-2\right)^{2}<0$, which implies $\rho_1>n-3$. By $f_1(n+\frac{2d}{n}-3)=\frac{8d^3}{n^3}+(\frac{4}{n}-\frac{20}{n^2}-2)d^2+(2n-2+\frac{20}{n})d-8$, we take $q_1(x)=\frac{8x^3}{n^3}+(\frac{4}{n}-\frac{20}{n^2}-2)x^2+(2n-2+\frac{20}{n})x-8$. Then $q_1'(x)=\frac{24}{n^3}x^2+(\frac{8}{n}-\frac{40}{n^2}-4)x+2n-2+\frac{20}{n}$. Since $q_1'(n-1)=10+\frac{24}{n^3}-\frac{8}{n^2}-\frac 4n-2n<0$, $q_1(d)\geq \min\{q_1(4),q_1(n-1)\}$ by $4\leq d \leq n-1$. Clearly, $q_1(n-1)=4n-8-\frac{8}{n^3}+\frac{4}{n^2}>0$ and $q_1(4)=8n-48+\frac{512}{n^3}-\frac{320}{n^2}+\frac{144}{n}>0$, and then $f_1(n+\frac{2d}{n}-3)=q_1(d)\geq\min\{q_1(4),q_1(n-1)\}>0.$	When $x\geq n-3$, $f_1'(x)=3x^2-2(n-4)x+7-n-2d\geq f_1'(n-3)=n^2-5n-2d+10\geq n^2-5n-2(n-1)+10=n^2-7n+12>0$. Then $f_1(x)$ is increasing in the interval $[n-3,+\infty)$, and $f_1(x)\geq f_1(n+\frac{2d}{n}-3)>0$ when $x\geq n+\frac{2d}{n}-3$. Then $\rho_1<n+\frac{2d}{n}-3$.
	
	By direct computation, we have $f_2(n+\frac dn-3)=\frac{d^3}{n^3}-(16+\frac{3}{n^2}+\frac{6}{n})d^2+(47+\frac{12}{n}+n)d-20-10n+2n^2.$ Let $q_2(x)=\frac{x^3}{n^3}-(16+\frac{3}{n^2}+\frac{6}{n})x^2+(47+\frac{12}{n}+n)x-20-10n+2n^2$. Then $q_2'(x)=\frac{3}{n^3}x^2-2(16+\frac{3}{n^2}+\frac{6}{n})x+47+\frac{12}{n}+n$. Since $\frac{16+\frac{3}{n^2}+\frac 6n}{\frac{3}{n^3}}=\frac{16n^3+3n+6n^2}{3}>n-1\geq d\geq\frac{n-2}{2}$, $q_2'(x)\leq q_2'(\frac{n-2}{2})=-15n+73+\frac{3}{n^3}+\frac{3}{n^2}+\frac{87}{4n}<0$ when $x\geq\frac{n-2}{2}$. Then $q_2(x)$ is decreasing in the interval $[\frac{n-2}{2},+\infty)$ and thus $q_2(d)\leq q_2(\frac{n-2}{2})=-\frac{3}{2}n^2+27n-\frac{63}{4n}-\frac{3}{2n^2}-\frac{1}{n^3}-\frac{573}{8}\leq-\frac{3}{2}n^2+27n-\frac{573}{8}$. 
	
	Let $q_3(x)=-\frac{3}{2}x^2+27x-\frac{573}{8}$. By $n\geq 15$, $q_3(n)\leq q_3(15)=-\frac{33}{8}<0$. Hence $f_2(n+\frac dn-3)=q_2(d)\leq q_2(\frac{n-2}{2})\leq q_3(n)\leq q_3(15)<0.$ Then $\rho_2>n+\frac dn-3$.
\end{proof}

\begin{lemma}\label {f*}
	Let $n$, $s$, $k$ be positive integers with $n>s+2k$, $G_k=K_s\vee(K_{n-s-2k}\cup2kK_1)$,  $H_2=K_{2(d-2)}\vee(K_{n-2d}\cup4K_1)$, $F_k(x)=x^3-(n-2k-2)x^2+(1-n+2k-2ks)x-4k^2s+2kns-2ks^2-2ks.$ Then $\rho(G_k)$ is the largest root of $F_k(x)=0$. In particular, if $\frac{n-2}{2}\leq d\leq n-1$ and $n\geq 15$, then $F_k(x)$ is increasing in the interval $[\rho(H_2),+\infty)$.	
\end{lemma}

\begin{proof}
	Considering the equitable partition $V(G_k)=V(K_s)\cup V(K_{n-s-2k})\cup V(2kK_1)$, the corresponding quotient matrix of $A(G_k)$ is $$B_k^*=\begin{pmatrix}
		s-1&n-s-2k&2k\\
		s&n-s-2k-1&0\\
		s&0&0\\
	\end{pmatrix},$$ and then the characteristic polynomial of the $B_k^*$ is equal to $F_k(x)$. In view of Lemma \ref{Spectra}, $\rho(G_k)$ is the largest root of $F_k(x)=0$.
	
	Let $\rho_2=\rho(H_2)$. By Lemma \ref{rho}, $\frac{n-2}{2}\leq d\leq n-1$ and $n\geq 15$, we have $\rho(H_2)>n+\frac dn-3>n-3$. Clearly, ${F_k}'(x)=3x^2-2(n-2k-2)x+1-n+2k-2ks$. By $\frac{n-2k-2}{3}<n-3<\rho_2$ and $n\geq2k+s+1\geq 2k+2$, we obtain $
	{F_k}'(x)\geq {F_k}'(\rho_2)> {F_k}'(n-3)=n^{2}+\left(-9+4 k\right) n-2 s k-10 k+16\geq n^{2}+\left(-9+4 k\right) n-2 (n-2k-1) k-10 k+16=4 k^{2}+\left(2 n-8\right) k+n^{2}-9 n+16\geq n^{2}-7 n+12>0$ when $x\geq\rho_2$. Then $F_k(x)$ is increasing in the interval $[\rho_2,+\infty)$.
\end{proof}

	\subsection{Proof of Theorem \ref{tt1}} 
	\hspace{1.5em}Assume the contrary that $G$ contains no spanning tree with leaf distance at least $d$. By Proposition \ref{p3}, there exists a nonempty subset $S\subseteq V(G)$ such that $i(G-S)\geq \frac{2|S|}{d-2}\text{ i.e. }i(G-S)\geq \Big\lceil{\frac{2|S|}{d-2}}\Big\rceil.$
	
	We choose a connected graph $H$ of order $n$ satisfying $V(H)=V(G)$, $\delta(H)\geq t$, $\alpha(H)\leq 5$ and $i(H-S)\geq \Big\lceil{\frac{2|S|}{d-2}}\Big\rceil$ so that its size is as large as possible. Then $|E(G)|\leq |E(H)|$. According to the choice of $H$, we see that the induced subgraph $H[S]$ and each connected component of $H-S$ are complete graphs, and $H=H[S]\vee(H-S)$.
	
	Firstly, we claim that there is at most one non-trivial connected component in $H-S$. Otherwise, we can obtain a new graph $H'$ by adding edges among all non-trivial connected components to derive a bigger non-trivial connected component. Obviously, $H$ is a proper spanning subgraph of $H'$ with $|E(G)|<|E(H)|$, but $H'$ satisfies $V(H')=V(G)$, $\delta(H')\geq t$, $\alpha(H')\leq 5$ and $i(H'-S)\geq \Big\lceil{\frac{2|S|}{d-2}}\Big\rceil$, which is a contradiction with the choice of $H$. 
	
	Let $i(H-S)=i$ and $|S|=s$ for short. Then $i\geq \lceil{\frac{2s}{d-2}}\rceil=\phi(s)$ by $i(H-S)\geq  \Big\lceil{\frac{2|S|}{d-2}}\Big\rceil$. Since $d-2\mid 2t$, we have $\phi(t)=\frac{2t}{d-2}$ and $s\phi(t)=s\cdot\frac{2t}{d-2}=t\cdot\frac{2s}{d-2}\leq t\phi(s)$. Clearly, $\phi(t)\leq\phi(s)$ by $t\leq \delta(G)\leq s$.
	
	We complete the proof by considering the following two cases.
	
	{\noindent\textbf{Case 1.}} $H-S$ has exactly one non-trivial connected component.
	
	In this case, let $F$ be the unique non-trivial connected component of $H-S$ and $V(F)=n_1\geq2$. Now we show $i=\lceil{\frac{2s}{d-2}}\rceil=\phi(s)$. 
	
	If $i\geq \phi(s)+1$, then we take a new graph $H''$ obtained from $H$ by joining each vertex of $F$ with one vertex in $V(H-S)\backslash V(F)$ by an edge. Then we have $$|E(H'')|=|E(H)|+n_1>|E(H)|,\text{ }\alpha(H'')\leq\alpha(H)\leq 5,$$ $$ \delta(H'')\geq\delta(H)\geq t ,\text{ } i(H''-S)=i-1\geq \phi(s),$$ which is a contradiction to the choice of $H$. Hence $i=\phi(s)$ by $i\geq \phi(s)$, and $H=K_s\vee(K_{n-s-i}\cup iK_1)$. Clearly, $n=s+i+n_1\geq s+i+2$, $|E(H)|=\binom{n-\phi(s)}{2}+s\phi(s)$ and $i=\phi(s)$. Let $w=\phi(s)-\phi(t)$. Then we have $$\begin{aligned}
	 &\binom{n-\phi(t)}{2}+t\phi(t)+4t-4-|E(H)|\\=&\frac{(2n-\phi(s)-\phi(t)-1)(\phi(s)-\phi(t))}{2}+t\phi(t)+4t-4-s\phi(s)\\ \geq&\frac{(\phi(s)+2s-\phi(t)+3)(\phi(s)-\phi(t))}{2}+t\phi(t)+4t-4-s\phi(s)\\=&\frac{(\phi(s)-\phi(t)+3)(\phi(s)-\phi(t))}{2}+t\phi(t)-s\phi(t)+4t-4\\\geq&\frac{(\phi(s)-\phi(t)+3)(\phi(s)-\phi(t))}{2}+t\phi(t)-t\phi(s)+4t-4\\=&
	 \frac{w(w+3)}{2}-tw+4t-4.
	\end{aligned}$$
	
	By $H=K_s\vee(K_{n-s-i}\cup iK_1)$ and $\alpha(H)\leq 5$, we have $\phi(s)=\alpha(H)-1\leq 4$, and $w=\phi(s)-\phi(t)\in\{0,1,2,3\}$ by $1\leq\phi(t)\leq\phi(s)$. Let $r_1(w)=\frac{w(w+3)}{2}-tw+4t-4.$ Thus $r_1(w)\geq\min\{r_1(w)|w=0,1,2,3\}\geq0$, which implies $|E(G)|\leq |E(H)|\leq\binom{n-\phi(t)}{2}+t\phi(t)+4t-4,$ a contradiction.
	
	{\noindent\textbf{Case 2.}} $H-S$ has no non-trivial connected component.
	
	In this case, we prove $i\leq \phi(s)+1$ firstly. If $i\geq \phi(s)+2$, then we take a new graph $H'''$ obtained from $H$ by adding an edge between two vertices in $V(H-S)$. Thus $i(H'''-S)=i-2\geq \phi(s)$, $\alpha(H''')\leq\alpha(H)\leq5$, $\delta(H''')\geq\delta(H)\geq t$ and $|E(H''')|=|E(H)|+1$. This contradicts to the choice of $H$, which implies $i=\phi(s)$ or $i=\phi(s)+1$ by $i\geq\phi(s)$.
	
	{\noindent\textbf{Subcase 2.1.}} $i=\phi(s)$.
	
	In this subcase, we have $H=K_s\vee iK_1$, $n=s+i$ and $|E(H)|=si+\binom{s}{2}$, $\alpha(H)=i=\phi(s)$ and $w=\phi(s)-\phi(t)\in\{0,1,2,3,4\}$ by $1\leq\phi(t)\leq\phi(s)=\alpha(H)\leq 5$.
	
	{\noindent\textbf{Subcase 2.1.1.}} $t=1$.
	
		Then $d=4$ by $d-2\mid 2t$ and $d\geq 4$, thus $\phi(t)=t=1$, $\phi(s)=s$ and $n=2s$.
		
		If $n\geq 7$, then $s\geq 4$, and $
		\binom{n-\phi(t)}{2}+t\phi(t)+4t-4-|E(H)|=\frac{(s-4)(s-1)}{2}\geq0,$ a contradiction.
		
		If $n=6$, then $s=3$, and $\binom{n-\phi(t)}{2}+t\phi(t)+4t-3-|E(H)|=1+\frac{(s-4)(s-1)}{2}=0$, a contradiction.
		
		{\noindent\textbf{Subcase 2.1.2.}} $t\geq2$.

	For $n\geq 5$, we have $$\begin{aligned}
	&\binom{n-\phi(t)}{2}+t\phi(t)+4t-4-|E(H)|\\=&\frac{(\phi(s)-\phi(t)-1)(\phi(s)-\phi(t))}{2}+t\phi(t)-s\phi(t)+4t-4\\ \geq&\frac{(\phi(s)-\phi(t)-1)(\phi(s)-\phi(t))}{2}+t\phi(t)-t\phi(s)+4t-4\\=&
	\frac{w(w-1)}{2}-tw+4t-4.
	\end{aligned}$$
	
	Then $r_2(w)\geq\min\{r_2(w)|w=0,1,2,3,4\}>0$, where $r_2(w)= \frac{w(w-1)}{2}-tw+4t-4.$ Thus $|E(G)|\leq|E(H)|<\binom{n-\phi(t)}{2}+t\phi(t)+4t-4$, a contradiction.
	
	{\noindent\textbf{Subcase 2.2.}} $i=\phi(s)+1$.
	
	In this subcase, $n=s+i$ and $H=K_s\vee iK_1$. By $|E(H)|=\binom{s}{2}+si$, we have $$\begin{aligned}
		&\binom{n-\phi(t)}{2}+t\phi(t)+4t-4-|E(H)|\\=&\frac{(\phi(s)-\phi(t))(\phi(s)-\phi(t)+1)}{2}+t\phi(t)-s\phi(t)+4t-4\\ \geq&\frac{(\phi(s)-\phi(t))(\phi(s)-\phi(t)+1)}{2}+t\phi(t)-t\phi(s)+4t-4\\=&
		\frac{w(w+1)}{2}-tw+4t-4.
	\end{aligned}$$
	
	Let $r_3(w)=\frac{w(w+1)}{2}-tw+4t-4$. It is easy to check $r_3(w)\geq0$ for $w\in\{0,1,2,3,4\}$. Hence $|E(G)|\leq|E(H)|\leq\binom{n-\phi(t)}{2}+t\phi(t)+4t-4<\binom{n-\phi(t)}{2}+t\phi(t)+4t-3$, a contradiction. 
	
	Combining the above arguments, $G$ has a spanning tree with leaf distance at least $d$.
	
	{\hfill $\blacksquare$ \par}

	\subsection{Proof of Theorem \ref{tt2}} 
	\hspace{1.5em}Suppose the contrary that $G$ contains no spanning tree with leaf distance at least $d$. By applying Theorem \ref{t5}, there exists $k$ such that $1\leq k\leq\frac{\alpha(G)}{2}<3$ and $\delta_{2k}(G)\leq k(d-2).$
	
	We choose a connected graph $H$ satisfying $V(H)=V(G)$, $\delta({H})\geq t$, $\alpha({H})\leq 5$ and $\delta_{2k}(H)\leq k(d-2)$ so that its size is as large as possible. Then $|E(G)|\leq|E(H)|$. By the definition of $\delta_{2k}(H)$, there is an independent set $I\subseteq V(H)$ of order $2k$ such that $\delta_{2k}(H)=|N_{H}(I)|$. According to the choice of ${H}$, we see that the induced subgraph ${H}[N_{H}(I)]$ and each connected component of ${H}-N_{H}(I)\cup I$ are complete graphs (if $V({H}-N_{H}(I)\cup I)\neq \emptyset$), and ${H}={H}[N_{H}(I)]\vee({H}-N_{H}(I))$.
	
	We claim that there is at most one non-trivial connected component in ${H}-N_{H}(I)$. Otherwise, we can obtain a new graph ${H}'$ by adding edges among all non-trivial connected components to derive a bigger non-trivial connected component. Obviously, ${H}$ is a proper spanning subgraph of ${H}'$ with $|E(H')|>|E(H)|$, but $H'$ satisfies $V(H')=V(G)$, $\delta({H'})\geq \delta (H)\geq t$, $\alpha({H'})\leq \alpha(H)\leq 5$ and $\delta_{2k}(H')=\delta_{2k}(H)\leq k(d-2)$, which is a contradiction with the choice of ${H}$. 
	
	We complete the proof by considering the following two cases.
	
	{\noindent\textbf{Case 1.}} ${H}-N_{H}(I)$ has exactly one non-trivial connected component.
	
	In this case, let $F$ be the unique non-trivial connected component of ${H}-N_{H}(I)$ and $V(F)=n_1\geq2$. Now we show $V({H}-N_{H}(I)-V(F))=I$.
	
	If $I\subsetneq V({H}-N_{H}(I)-V(F))$, then let ${H}''$ be a new graph obtained from ${H}$ by joining each vertex of $F$ with one vertex in $V({H}-N_{H}(I)-V(F))-I$ by an edge. Then we have $|E({H}'')|=|E({H})|+n_1>|E({H})|$, $\delta({H}'')\geq\delta({H})\geq t$, $\alpha({H}'')\leq\alpha({H})\leq5$ and $\delta_{2k}({H}'')=|N_{{H}''}(I)|=|N_{H}(I)|=\delta_{2k}({H})\leq k(d-2)$, which contradicts the choice of ${H}$. Hence $V({H}-N_{H}(I)-V(F))=I$ and ${H}=K_{\delta_{2k}({H})}\vee(K_{n-\delta_{2k}({H})-2k}\cup 2kK_1)$. 
	
	By $\delta_{2k}({H})\leq k(d-2)$, we have $$|E(G)|\leq|E({H})|\leq|E(K_{k(d-2)}\vee(K_{n-kd}\cup2kK_1))|=\binom{n-2k}{2}+2k^2(d-2).$$ 
	
	For $t>d-2$, then $k(d-2)\geq\delta_{2k}({H})\geq\delta({H})\geq t>d-2$. Then $k=2$ by $1\leq k\leq 2$, and $|E(G)|\leq\binom{n-4}{2}+8(d-2)$, a contradiction.
	
	For $t\leq d-2$ and $k=1$, $|E(G)|\leq\binom{n-2}{2}+2(d-2)$, a contradiction. 
	
	For $t\leq d-2$ and $k=2$, we have $
		\binom{n-2}{2}+2(d-2)-|E(G)|\geq\binom{n-2}{2}+2(d-2)-\binom{n-4}{2}-8(d-2)>0$ by $n\geq 3d$,  then $|E(G)|<\binom{n-2}{2}+2(d-2)$, a contradiction.

	{\noindent\textbf{Case 2.}} ${H}-N_{H}(I)$ has no non-trivial connected component.
	
	In this case, we have ${H}=K_{\delta_{2k}({H})}\vee(n-\delta_{2k}({H}))K_1$. We claim $n-\delta_{2k}({H})\leq 2k+1$. Otherwise, if $n-\delta_{2k}({H})\geq 2k+2$, then we can obtain a new graph obtained from ${H}$ by adding an edge in $V({H})-N_{H}(I)-I$, which is a contradiction to the choice of ${H}$. Hence $n-\delta_{2k}({H})\leq 2k+1$. Then $3d\leq n\leq\delta_{2k}({H})+2k+1\leq k(d-2)+2k+1=kd+1\leq 2d+1$, say, $d\leq 1$, a contradiction with $d\geq 4$. 
	
	Combining the above arguments, $G$ has a spanning tree with leaf distance at least $d$.
	
	{\hfill $\blacksquare$ \par}
	
	\subsection{Proof of Theorem \ref{tt3}} 
	\hspace{1.5em}Assume the contrary that $G$ contains no spanning tree with leaf distance at least $d$. By Corollary \ref{c3}, there exists $k$ such that $1\leq k\leq\frac{\alpha(G)}{2}$ and $\delta_{2k}(G)\leq k(d-2).$
	
	We choose a connected graph $H$ of order $n$ satisfying $V(H)=V(G)$ and $\delta_{2k}(H)\leq k(d-2)$ so that its size is as large as possible. Then $|E(G)|\leq|E(H)|$. By the definition of $\delta_{2k}(H)$, there is an independent set $I\subseteq V(H)$ of order $2k$ such that $\delta_{2k}(H)=|N_{H}(I)|$, denoted by $\delta_{2k}$. According to the choice of ${H}$, we see that the induced subgraph ${H}[N_{H}(I)]$ and each connected component of ${H}-N_{H}(I)\cup I$ are complete graphs (if $V({H}-N_{H}(I)\cup I)\neq \emptyset$), and ${H}={H}[N_{H}(I)]\vee({H}-N_{H}(I))$.
	
	Similar to the proof of Theorem \ref{tt2}, we claim there is at most one non-trivial connected component in $H-N_{H}(I)$, then we consider the following two possible cases.
	
	{\noindent\textbf{Case 1.}} ${H}-N_{H}(I)$ has exactly one non-trivial connected component.
	
	In this case, let $F$ be the unique non-trivial connected component of ${H}-N_{H}(I)$ and $V(F)=n_1\geq2$. It is easy to check $V({H}-N_{H}(I)-V(F))=I$. Then ${H}=K_{\delta_{2k}}\vee(K_{n-\delta_{2k}-2k}\cup2kK_1)$ and $n_1=n-\delta_{2k}-2k\geq 2$ (i.e. $\delta_{2k}\leq n-2k-2$). Now we show $\delta_{2k}=\begin{cases}
		k(d-2),& \text{if $n\geq kd+2$}\\
		n-2k-2,& \text{if $n\leq kd+1$}
	\end{cases}$ and ${H}=\begin{cases}
	K_{k(d-2)}\vee(K_{n-kd}\cup2kK_1),& \text{if $n\geq kd+2$}\\
	K_{n-2k-2}\vee(K_{2}\cup2kK_1),& \text{if $n\leq kd+1$}
	\end{cases}$. 
	
	If $\delta_{2k}\leq k(d-2)-1$ and $n\geq kd+2$, then we take a new graph ${H}'$ obtained from ${H}$ by joining each vertex of $I$ with one vertex in $V(F)$. Then we have $|E({H}')|=|E({H})|+2k$ and $\delta_{2k}({H}')\leq|N_{{H}'}(I)|=|N_{{H}}(I)|+1\leq k(d-2),$ which is a contradiction with the choice of ${H}$. Thus $\delta_{2k}=k(d-2)$ by 
	$\delta_{2k}\leq k(d-2)$ and ${H}=K_{k(d-2)}\vee(K_{n-kd}\cup2kK_1)$. 
	
	If $\delta_{2k}\leq n-2k-3$ and $n\leq kd+1$, then we take a new graph ${H}''$ obtained from ${H}$ by joining each vertex of $I$ with one vertex in $V(F)$. Then we have $|E({H}')|=|E({H})|+2k$ and $\delta_{2k}({H}'')\leq|N_{{H}''}(I)|=|N_{{H}}(I)|+1\leq n-2k-2\leq k(d-2)-1,$ which is a contradiction with the choice of ${H}$. Thus $\delta_{2k}=n-2k-2$ by 
	$\delta_{2k}\leq n-2k-2$ and ${H}=K_{n-2k-2}\vee(K_{2}\cup2kK_1)$.
	
	{\noindent\textbf{Subcase 1.1.}} $n\geq kd+2$.
	
	In this subcase, we have $k\in\begin{cases}
	\{1,2\},& \text{if $\frac n3\leq d\leq \frac{n-1}{2}$}\\
		\{1\},& \text{if $\frac n2\leq d\leq n-1$}
	\end{cases}$, and ${H}=K_{k(d-2)}\vee(K_{n-kd}\cup2kK_1)$. Thus $|E(H)|=\binom{n-2k}{2}-2k^2(d-2)$. 
	
	{\noindent\textbf{Subcase 1.1.1.}} $\frac n3\leq d\leq \frac{n-1}{2}$.
	
	For $\max\{5,\frac n3\}\leq d \leq \frac{n-3}{2}$, we have $
	r_4(k)=\binom{n-6}{2}+6(n-6)-|E({H})|=\binom{n-6}{2}+6(n-6)-\binom{n-2k}{2}-2k^2(d-2)\geq\min\{r_4(1),r_4(2)\}> 0$, which implies $|E(G)|<\binom{n-6}{2}+6(n-6)$ by $|E({H})|\geq |E(G)|$, a contradiction.
	
	For $\frac{n-2}{2}\leq d \leq\frac{n-1}{2}$, we have 
	$r_5(k)=\binom{n-4}{2}+8(d-2)-|E(H)|\geq\min\{r_5(1),r_5(2)\}\geq 0$ by $d\geq 5$, which implies $|E(G)|\leq|E({H})|<\binom{n-4}{2}+8(d-2)$, a contradiction.
	
	{\noindent\textbf{Subcase 1.1.2.}} $\frac n2\leq d\leq n-1$.
	
	In this case,  $k=1$, $H=K_{d-2}\vee(K_{n-d}\cup 2K_1)$ and $|E({H})|=\binom{n-2}{2}+2(d-2)$.
	
	For $\frac{n}{2}\leq d \leq n-3$, we have $\binom{n-4}{2}+4(n-4)-|E(H)|=2n-2d-5\geq2(d+3)-2d-5=1>0$, which implies $|E(G)|\leq|E({H})|<\binom{n-4}{2}+4(n-4)$, a contradiction.
	
	For $n-2\leq d \leq n-1$, we have $|E(G)|\leq|E({H})|$ immediately, a contradiction.
	
	{\noindent\textbf{Subcase 1.2.}} $n\leq kd+1$.
	
	In this subcase, we have $k\geq\begin{cases}
		 3,& \text{if $\frac n3\leq d\leq \frac{n-3}{2}$}\\
		 2,& \text{if $\frac{n-2}{2}\leq d\leq n-3$}\\
		 1,& \text{if $n-2\leq d\leq n-1$}
	\end{cases}$, and ${H}=K_{n-2k-2}\vee(K_{2}\cup2kK_1)$. Thus $|E(H)|=\binom{n-2k}{2}-2k(n-2k-2)$.
	
	{\noindent\textbf{Subcase 1.2.1.}} $\frac n3\leq d\leq \frac{n-3}{2}$.
	
	In this subcase, we have $\binom{n-6}{2}+6(n-6)-|E({H})|=2k^2+3k-15>0$, which implies $|E(G)|<\binom{n-6}{2}+6(n-6)$ by $|E(G)|\leq|E(H)|$, a contradiction.
	
	{\noindent\textbf{Subcase 1.2.2.}} $\frac{n-2}{2}\leq d\leq n-3$.
	
	For $\frac{n-2}{2}\leq d\leq \frac{n-1}{2}$, we have $\binom{n-4}{2}+8(d-2)-|E(H)|=2k^2+3k-4n-6+8d\geq8-4n+8d\geq 0,$ which implies $|E(G)|\leq |E(H)|\leq\binom{n-4}{2}+8(d-2)$, a contradiction.
	
	For $\frac{n}{2}\leq d\leq n-3$, we have $\binom{n-4}{2}+4(n-4)-|E(H)|=2k^2+3k-6>0,$ which implies $|E(G)|\leq |E(H)|<\binom{n-4}{2}+4(n-4)$, a contradiction. 
	
	{\noindent\textbf{Subcase 1.2.3.}} $n-2\leq d\leq n-1$.
	
	In this case, $\binom{n-2}{2}+2(d-2)-|E(H)|=2k^2+3k-2n-1+2d\geq2k^2+3k-5\geq0,$ which implies $|E(G)|\leq\binom{n-2}{2}+2(d-2)$, a contradiction.

	{\noindent\textbf{Case 2.}} ${H}-N_{H}(I)$ has no non-trivial connected component.
	
	In this case, we have ${H}=K_{\delta_{2k}}\vee(n-\delta_{2k})K_1$. It is easy to check $2k\leq n-\delta_{2k}\leq 2k+1$.
	
	{\noindent\textbf{Subcase 2.1.}} $\delta_{2k}=n-2k$.
	
	In this subcase, we have ${H}=K_{n-2k}\vee2kK_1$ and $|E(G)|\leq|E(H)|=\binom{n-2k}{2}+2k(n-2k)$. 
	
	{\noindent\textbf{Subcase 2.1.1.}} $\frac n3\leq d\leq\frac{n-1}{2}$.
	
	By $2d+1\leq n=2k+\delta_{2k}\leq kd$, we have $k\geq 3$. 
	
	For $\frac n3\leq d\leq\frac{n-3}{2}$, we have $\binom{n-6}{2}+6(n-6)-|E(H)|=2k^2-k-15\geq0,$ which implies $|E(G)|\leq \binom{n-6}{2}+6(n-6)$ by $|E(H)|\geq |E(G)|$, a contradiction.

	For $\frac {n-2}{2}\leq d\leq\frac{n-1}{2}$, $\binom{n-4}{2}+8(d-2)-|E(H)|=2k^2-k-4n-6+8d\geq9-4n+8\cdot\frac{n-2}{2}>0,$ which implies $|E(G)|\leq \binom{n-4}{2}+8(d-2)$ by $|E(H)|\geq |E(G)|$, a contradiction.
	
	{\noindent\textbf{Subcase 2.1.2.}} $\frac {n}{2}\leq d\leq n-1$.
	
	By $d+1\leq n=\delta_{2k}+2k\leq kd$, we have $k\geq2$.
	
	For $\frac {n}{2}\leq d\leq n-3$, we have $\binom{n-4}{2}+4(n-4)-|E(G)|\geq2k^2-k-6\geq0,$ a contradiction.
	
	For $n-2\leq d\leq n-1$, $\binom{n-2}{2}+2(d-2)-|E(H)|=2k^2-k-5>0,$ a contradiction.
	
	{\noindent\textbf{Subcase 2.2.}} $\delta_{2k}=n-2k-1$.
	
	In this subcase, we have ${H}=K_{n-2k-1}\vee(2k+1)K_1$, and $|E(G)|\leq|E(H)|=\binom{n-2k-1}{2}+(2k+1)(n-2k-1)<|E(K_{n-2k}\vee2kK_1)|$. 
	
	{\noindent\textbf{Subcase 2.2.1.}} $\frac n3\leq d\leq\frac{n-3}{2}$.
	
	By $2d+3\leq n=\delta_{2k}+2k+1\leq kd+1$, we have $k\geq3$ and then $\binom{n-6}{2}+6(n-6)-|E(G)|>0$ by Subcase 2.1.1 and $|E(G)|<|E(K_{n-2k}\vee2kK_1)|$, a contradiction.
	
	{\noindent\textbf{Subcase 2.2.2.}} $\frac {n-2}{2}\leq d\leq n-3$.
	
	By $d+3\leq n=\delta_{2k}+2k+1\leq kd+1$, we have $k\geq2$. 
	
	For $\frac {n-2}{2}\leq d\leq \frac{n-1}{2}$ and $k=2$, we have $n=2d+1$ and $\binom{n-4}{2}+8(d-2)-|E(H)|=2k^2+k-10\geq0,$ which implies $|E(G)|\leq |E(H)|\leq\binom{n-4}{2}+8(d-2)$, a contradiction.
	
	For $\frac {n-2}{2}\leq d\leq \frac{n-1}{2}$ and $k\geq 3$, we have $\binom{n-4}{2}+8(d-2)-|E(G)|>0$ by Subcase 2.1.1, a contradiction.
	
	For $\frac {n}{2}\leq d\leq n-3$ and $k\geq 2$, we have $\binom{n-4}{2}+4(n-4)-|E(G)|>0$ by Subcase 2.1.2, a contradiction.
	
	{\noindent\textbf{Subcase 2.2.3.}} $n-2\leq d\leq n-1$.
	
	By $d+1\leq n\leq kd+1$, we have $k\geq1$. 
	
	If $k=1$, then $n=d+1$ and ${H}=K_{n-3}\vee3K_1=K_{d-2}\vee(K_{n-d}\cup 2K_1)$, which implies $|E(G)|\leq|E({H})|=\binom{n-2}{2}+2(d-2)$, a contradiction. 
	
	If $k\geq2$, then $\binom{n-2}{2}+2(d-2)-|E(G)|>0$ by Subcase 2.1.2, a contradiction.
	
	Combining the above arguments, $G$ has a spanning tree with leaf distance at least $d$.
	
	{\hfill $\blacksquare$ \par}
	
	\subsection{Proof of Theorem \ref{tt4}} 
	\hspace{1.5em}Let $H=K_t\vee(K_{n-t-\phi(t)}\cup\phi(t)K_1)$, $G$ be a connected graph of order $n\geq 6t+5$ with $\delta(G)\geq t$, $\alpha(G)\leq 5$, and $\rho(G)\geq\rho(H)$. Suppose to the contrary that $G$ contain no spanning tree with leaf distance at least $d$. In light of Proposition \ref{p3}, there exists a nonempty subset $S\subseteq V(G)$ such that $i(G-S)\geq\frac{2|S|}{d-2}.$ For convenience, let $|S|=s$ and so $i(G-S)\geq\lceil{\frac{2s}{d-2}}\rceil=\phi(s)$. Then $G^*=\begin{cases}K_s\vee(K_{n-s-\phi(s)}\cup\phi(s)K_1),&\text{if $n\geq s+\phi(s)+1$}\\ K_s\vee \phi(s)K_1,&\text{if $n=s+\phi(s)$} \end{cases}$ is the graph with the maximum size such that $V(G^*)=V(G)$, $i(G^*-S)\geq\phi(s)$, $\alpha(G^*)\leq\alpha(G)\leq 5$ and $\delta(G^*)=s\geq \delta(G)\geq t$. Clearly, $G^*$ has a Hamilton path by $\phi(s)\leq s$, then $G$ has a spanning tree with leaf distance at least $d$, thus $G\ncong G^*$, and $G$ is a proper spanning subgraph of $G^*$. Therefore, $\rho(G^*)>\rho(H)$ by Lemma \ref{rq1} and $\rho(G)\geq\rho(H)$.

	%In what follows, we proceed to prove Theorem \ref{tt4} by considering the following two cases.
	
	If $s=t$, then $G^*=H$, thus $\rho(G^*)=\rho(H)$, a contradiction with $\rho(G^*)>\rho(H)$. Next, we have $s\geq t+1$ and consider the following two cases since $n\geq s+\phi(s)$. 
%	{\noindent\textbf{Case 1.}} $s=t$.
	
	%In this case, $G^*=H$, then $\rho(G^*)=\rho(H)$, a contradiction with $\rho(G^*)>\rho(H)$.
	
	%{\noindent\textbf{Case 2.}} $s\geq t+1$.
	
	%Now we show $\rho(G^*)<\rho(H)$ by considering the following two subcases since $n\geq s+\phi(s)$. 
	
	%{\noindent\textbf{Subcase 2.1.}} $n=s+\phi(s)$.
	
	{\noindent\textbf{Case 1.}} $n=s+\phi(s)$.
	
	In this subcase, we have $G^*=K_s\vee \phi(s)K_1$. Clearly, $\phi(s)=\alpha(G^*)\leq\alpha(G)\leq5$. By $n\geq6t+5$ and $s\geq\phi(s)$, we have $5\geq\phi(s)=\lceil{\frac{2s}{d-2}}\rceil\geq\frac{2s}{d-2}\geq\frac{s+\phi(s)}{d-2}\geq\frac{6t+5}{d-2}>\frac{6t}{d-2}=3\phi(t)$. Then $\phi(t)\leq 1$ and thus $\phi(t)=\frac{2t}{d-2}=1$ and $t=\frac{d-2}{2}$ by $d-2\mid 2t$. Therefore, $s\leq \frac{\phi(s)(d-2)}{2}\leq\frac{5(d-2)}{2}=5t$, and $6t+5\leq n=s+\phi(s)\leq5t+5$, a contradiction.
	
	%{\noindent\textbf{Subcase 2.2.}} $n\geq s+\phi(s)+1$.
	
	{\noindent\textbf{Case 2.}} $n\geq s+\phi(s)+1$.

	In this subcase, we have $G^*=K_s\vee(K_{n-s-\phi(s)}\cup \phi(s)K_1)$. Consider the partition $V(G^*)=V(K_s)\cup V(K_{n-s-\phi(s)})\cup V(\phi(s)K_1)$. The corresponding quotient matrix of $A(G^*)$ equals $$B^*=\begin{pmatrix}
		s-1 & n-s-\phi(s) & \phi(s)\\
		s & n-s-\phi(s)-1 & 0\\
		s & 0 & 0\\
	\end{pmatrix}.$$ Hence, we derive the characteristic polynomial of $B^*$ as $f^*(x)=x^3-(n-\phi(s)-2)x^2+(\phi(s)+1-s\phi(s)-n)x+ns\phi(s)-s^2\phi(s)-s\phi(s)^2-s\phi(s).$ Note that the partition $V(G^*)=V(K_s)\cup V(K_{n-s-\phi(s)})\cup V(\phi(s)K_1)$ is equitable, then $\rho^*=\rho(G^*)$ is the largest root of $f^*(x)=0$ by Lemma \ref{Spectra}. Similarly, $\rho=\rho(H)$ is the largest root of $f(x)=0$, where $f(x)=x^3-(n-\phi(t)-2)x^2+(\phi(t)+1-t\phi(t)-n)x+nt\phi(t)-t^2\phi(t)-t\phi(t)^2-t\phi(t).$ 
	
	Since $K_{n-\phi(t)}\cup \phi(t)K_1$ is a proper spanning subgraph of $H$, we have $\rho=\rho(H)>\rho(K_{n-\phi(t)}\cup \phi(t)K_1)=n-\phi(t)-1$ by Lemma \ref{rq1}. By direct computation, we have \begin{equation}\label{6.2}
		\begin{aligned}
			f^*(\rho)=f^*(\rho)-f(\rho)=g_1(\rho),
		\end{aligned}
	\end{equation} where $g_1(x)=[\phi(s)-\phi(t)]x^2+[\phi(s)(1-s)-\phi(t)(1-t)]x+[s\phi(s)-t\phi(t)](n-1)-s\phi(s)[s+\phi(s)]+t\phi(t)[t+\phi(t)]$. Since $d\geq4$, $\phi(s)=\lceil{\frac{2s}{d-2}}\rceil\geq\lceil{\frac{2t+2}{d-2}}\rceil=\phi(t)+\lceil{\frac{2}{d-2}}\rceil=\phi(t)+1$, we have $s\geq \phi(s)\geq \phi(t)+1$ and $\phi(s)-\phi(t)=\lceil{\frac{2s}{d-2}}\rceil-\frac{2t}{d-2}\geq\frac{2(s-t)}{d-2}$. Then $\frac{\phi(t)(1-t)-\phi(s)(1-s)}{2[\phi(s)-\phi(t)]}=\frac{[\phi(s)-\phi(t)](s-1)+\phi(t)(s-t)}{2[\phi(s)-\phi(t)]}\leq\frac{s-1}{2}+\frac{\phi(t)(s-t)}{\frac{4(s-t)}{d-2}}=\frac{s+t-1}{2}<s<n-\phi(t)-1$ by $n-\phi(t)-1>n-\phi(s)-1\geq s$ and thus $g_1(x)$ is increasing in the interval $[n-\phi(t)-1,+\infty)$. 
	
	By $\rho>n-\phi(t)-1$, we have \begin{equation}\label{6.3}
	\begin{aligned}g_1(\rho)>g_1(n-\phi(t)-1)=&-\phi(s)s^2+[\phi(s)\phi(t)-\phi(s)^2]s-\phi(t)^2-\phi(t)^3-n\phi(s)\\&+n\phi(t)+n^2\phi(s)-n^2\phi(t)+t^2\phi(t)+\phi(s)\phi(t)\\&+\phi(s)\phi(t)^2+2\phi(t)^2n-2n\phi(s)\phi(t).\end{aligned}
	 \end{equation}
	 
	 Let $h_1(x)=-\phi(s)x^2+[\phi(s)\phi(t)-\phi(s)^2]x-\phi(t)^2-\phi(t)^3-n\phi(s)+n\phi(t)+n^2\phi(s)-n^2\phi(t)+t^2\phi(t)+\phi(s)\phi(t)+\phi(s)\phi(t)^2+2\phi(t)^2n-2n\phi(s)\phi(t).$ By $\phi(s)t\geq \frac{2st}{d-2}=\phi(t)s$, we have $\frac{\phi(s)\phi(t)-\phi(s)^2}{2\phi(s)}=\frac{\phi(t)-\phi(s)}{2}<0<t+1\leq s\leq\frac{\phi(s)t}{\phi(t)}$, and then \begin{equation}\label{6.4}
	 \begin{aligned}
	 h_1(s)\geq h_1\left(\frac{\phi(s)t}{\phi(t)}\right)=\left[\phi(s)-\phi(t)\right]h_2(\phi(s),\phi(t),n),
	 \end{aligned}\end{equation} where $h_2(\phi(s),\phi(t),n)= n^{2}+\frac{[-2 \phi(t)^{3}-\phi(t)^{2}]  n}{\phi(t)^{2}}+\frac{\phi(t)^{4}+\phi(t)^{3}-t^{2} \phi(t)^{2}-t \phi(s) \left(t+\phi(s)\right) \phi(t)-t^{2} \phi(s)^{2} }{\phi(t)^{2}}$. Since $\phi(t)+1\leq \phi(s)\leq5$, $\phi(t)\geq1$ and $n\geq6t+5$, $h_2(\phi(s),\phi(t),n)\geq\min\{h_2(\phi(s),\phi(t),n) \mid \phi(t)+1\leq \phi(s)\leq5,\phi(t)\geq1\}>0$. For example, $h_2(5,1,n)=4n^2-12n-124t^2-100t+8\geq h_2(5,1,6t+5)=20t^2+68t+48>0$ due to $n\geq6t+5$.
 
 	Combining this with (\ref{6.2}), (\ref{6.3}) and (\ref{6.4}), we infer $f^*(\rho)=g_1(\rho)>g_1(n-\phi(t)-1)=h_1(s)\geq h_1(\frac{\phi(s)t}{\phi(t)})=[\phi(s)-\phi(t)]h_2(\phi(s),\phi(t),n)>0.$
 	
 	Now we proceed to prove ${f^*}'(x)>0$ in the interval $[\rho,+\infty)$. By a simple calculation, we know that ${f^*}'(x)=3x^2-2(n-\phi(s)-2)x+(\phi(s)+1-s\phi(s)-n)$. Since $\frac{n-\phi(s)-2}{3}<n-\phi(t)-1<\rho$, ${f^*}'(x)$ is increasing in the interval $[n-\phi(t)-1,+\infty)$. 
 	
 	When $x\geq\rho$, ${f^*}'(x)\geq {f^*}'(\rho)>{f^*}'(n-\phi(t)-1)=[n+2\phi(s)-3\phi(t)][n-\phi(t)-1]+\phi(s)-\phi(t)-s\phi(s)\geq[s+3\phi(s)-3\phi(t)+1][s+\phi(s)-\phi(t)]+\phi(s)-\phi(t)-s\phi(s)> s(s+1)+\phi(s)-\phi(t)-s\phi(s)=s(s-\phi(s)+1)+\phi(s)-\phi(t)>0$ by $n\geq s+\phi(s)+1$ and $s\geq \phi(s)>\phi(t)$.
 	
 	 Hence ${f^*}(x)$ is increasing in the interval $[\rho,+\infty)$, which implies ${f^*}(x)\geq {f^*}(\rho)>0$ when $x\geq\rho$. Then $\rho(G^*)=\rho^*<\rho=\rho(H)$, a contradiction.
 	
 	Combining Case 1 and Case 2, $G$ has a spanning tree with leaf distance at least $d$. {\hfill $\blacksquare$ \par}
	
	\subsection{Proof of Theorem \ref{tt5}} 
	\hspace{1.5em}Let $H_l=K_{l(d-2)}\vee(K_{n-ld}\cup2lK_1)$, $G$ be a connected graph of order $n\geq 3d\geq 12$ with $\delta(G)\geq t$, $\alpha(G)\leq 5$ and \begin{equation}\label{f1}
		\rho(G)\geq \begin{cases}
		\rho(H_1),& \text{if $t\leq d-2$}\\
		\rho(H_2),& \text{if $t>d-2$}
	\end{cases}.\end{equation} 

	Suppose to the contrary that $G$ contains no spanning tree with leaf distance at least $d$. By applying Theorem \ref{t5}, there exists $k$ such that $1\leq k\leq\frac{\alpha(G)}{2}<3$, say, $1\leq k\leq 2$, and $\delta_{2k}(G)\leq k(d-2).$ Then $H_k$ is the graph with maximum size such that $\delta_{2k}(H_k)\leq k(d-2)$. Clearly, $H_k$ has a spanning tree with leaf distance at least $d$, then $G\ncong H_k$ and thus $G$ is a proper spanning subgraph of $H_k$. Therefore, Lemma \ref{rq1} implies that \begin{equation}\label{f2}
		\rho(G)<\rho(H_k).
	\end{equation} 

	If $t>d-2$, then $d-2<t\leq \delta(G)\leq\delta_{2k}(G)\leq k(d-2)$ and so $k=2$, which implies $\rho(G)<\rho(H_2)=\rho(K_{2(d-2)}\vee(K_{n-2d}\cup4K_1)),$ a contradiction with (\ref{f1}).
	
	Now we consider the following two cases for $t\leq d-2$.
		
	{\noindent\textbf{Case 1.}} $k=1$.
	
	In this case, we have $\rho(G)<\rho(H_k)=\rho(H_1)$ by (\ref{f2}), a contradiction with (\ref{f1}).
	
	{\noindent\textbf{Case 2.}} $k=2$.
	
	In this case, we have $\rho(G)<\rho(H_2)$ by (\ref{f2}). It suffices to show that $\rho(H_2)<\rho(H_1)$ as follows. By Lemma \ref{rho}, $\rho_l=\rho(H_l)$ is the largest root of $f_l(x)$ defined in Lemma \ref{rho} and $\rho_1>n-3$. By direct calculation, we derive \begin{equation}\label{6.6}
			f_2(\rho_1)=f_2(\rho_1)-f_1(\rho_1)=2g_2(\rho_1),
	\end{equation} where $g_2(x)=x^{2}+\left(-3 d+7\right) x-7 d^{2}+\left(11+3 n\right) d-6 n+6$. By $n\geq 3d$ and $\frac{3d-7}{2}<n-3<\rho_1$, we have \begin{equation}\label{6.7}
	\begin{aligned}
		g_2(\rho_1)>g_2(n-3)=-7 d^{2}+20 d+n^{2}-5 n-6.
	\end{aligned}
	\end{equation} 
	
	Let $h_3(x)=-7 x^{2}+20 x+n^{2}-5 n-6$. Clearly, $h_3(x)$ is decreasing within the interval $[4,\frac n3]$. Then \begin{equation}\label{6.9}
	h_3(d)\geq h_3(\frac n3)=\frac{2n^2}{9}+\frac{5n}{3}-6>0.
	\end{equation}
	
	 By (\ref{6.6}), (\ref{6.7}) and (\ref{6.9}), we have $f_2(\rho_1)=2g_2(\rho_1)>2g_2(n-3)=2h_3(d)>0.$
	
	Now we proceed to prove ${f_2}'(x)>0$ in the interval $[\rho_1,+\infty)$. By a simple calculation, we have ${f_2}'(x)=3 x^{2}+ \left(12-2n\right) x-8 d-n+21$. Since $\frac{n-6}{3}<n-3<\rho_2$, ${f_2}'(x)\geq {f_2}'(\rho_1)>{f_2}'(n-3)=n^{2}-8 d-n+12\geq9 d^{2}-11 d+12>0$ by $n\geq3d\geq 12$ and $x\geq\rho_1$. Hence $f_2(x)$ is increasing in the interval $[\rho_1,+\infty)$, which yields $f_2(x)\geq f_2(\rho_1)>0$ when $x\geq\rho_1$. Then $\rho(G)<\rho(H_2)=\rho_2<\rho_1=\rho(H_1)$ by (\ref{f2}), a contradiction with (\ref{f1}).
	
	Combining the above arguments, $G$ has a spanning tree with leaf distance at least $d$.
	
	{\hfill $\blacksquare$ \par}
	
	\subsection{Proof of Theorem \ref{tt6}} 
	\hspace{1.5em}Let $H_l=K_{l(d-2)}\vee(K_{n-ld}\cup2lK_1)$, $R=K_{n-4}\vee4K_1$, $G$ be a connected graph of order $n\geq 15$ with (\ref{th 1.13}) holds. It is easy to check that $H_l$ has a spanning tree with leaf distance at least $d$. We suppose to the contrary that $G$ contain no spanning tree with leaf distance at least $d$ and then $G\ncong H_l$. By applying Corollary \ref{c3}, there exists $k$ such that $1\leq k\leq\frac{\alpha(G)}{2}$ and $\delta_{2k}(G)\leq k(d-2).$ Let $s=\delta_{2k}(G)$. Then $s\leq k(d-2)$ and $G$ is a spanning subgraph of $G_k$, where $G_k= K_{s}\vee(K_{n-s-2k}\cup2kK_1)$. By Lemma \ref{rq1}, we admit \begin{equation}\label{f3}
		\rho(G)\leq\rho(G_k),
	\end{equation} and the equality of (\ref{f3}) holds if and only if $G=G_k$.

	By Lemma \ref{rho}, $\rho_l=\rho(H_l)$ is the largest root of $f_l(x)$ defined in Lemma \ref{rho}. Moreover, if $4\leq d\leq n-1$, then $n-3<\rho_1<n+\frac{2d}{n}-3$; and if $n\geq15$ and $\frac{n-2}{2}\leq d\leq n-1$, then $\rho_2>n+\frac dn-3$.
		
	Now we proceed to prove Theorem \ref{tt6} by considering three cases.

	{\noindent\textbf{Case 1.}} $\frac{n-2}{2}\leq d \leq\frac{n-1}{2}$.
	
	In this case, we verify $\rho(G_k)\leq\rho(H_2)$. It suffices to consider the following two subcases by $n\geq s+2k$.
	
	{\noindent\textbf{Subcase 1.1.}} $n\geq s+2k+1$.
	
	{\noindent\textbf{Subcase 1.1.1.}} $k\in\{1,2\}$.
	
	Since $n\geq 2d+1$ and $k\in\{1,2\}$, $G_k$ is a spanning subgraph of $H_k$. By Lemma \ref{rq1}, $\rho(G_k)\leq\rho(H_k)$ with equality if and only if $G_k=H_k$. Now we show that $\rho(H_1)<\rho(H_2)$. 
	
	By direct calculation, we obtain \begin{equation}\label{3.11}f_2(\rho_1)=f_2(\rho_1)-f_1(\rho_1)=2g_3(\rho_1),
	\end{equation} where $g_3(x)=x^2+(7-3d)x+11d-7d^2+3dn-6n+6$. By $\frac{3d-7}{2}<n-3<\rho_1<n+\frac{2d}{n}-3$, we have \begin{equation}\label{3.12}
	\begin{aligned}
		g_3(\rho_1)<g_3(n+\frac{2d}{n}-3)=(\frac{4}{n^2}-\frac 6n-7)d^2+(24+\frac 2n)d-6-5n+n^2.
	\end{aligned}
	\end{equation}
	
	Let $p_1(x)=(\frac{4}{n^2}-\frac 6n-7)x^2+(24+\frac 2n)x-6-5n+n^2$. Since $\frac{-12-\frac{1}{n}}{\frac{4}{n^2}-\frac 6n-7}=\frac{12n^2+n}{7n^2+6n-4}<4<\frac{n-2}{2}\leq d\leq n-1$, we have \begin{equation}\label{3.13}
	\begin{aligned}
		p_1(d)\leq p_1(\frac{n-2}{2})=-\frac{3n^2}{4}+\frac{25n}{2}-29+\frac{4}{n^2}-\frac{12}{n}<-\frac{3n^2}{4}+\frac{25n}{2}-29=p_2(n),
	\end{aligned}
	\end{equation} where $p_2(x)=-\frac{3x^2}{4}+\frac{25x}{2}-29$. Clearly, $p_2(x)$ is decreasing in the interval $[14,+\infty)$ and thus $p_2(n)\leq p_2(14)=-1$ by $n\geq 14$. 
	
	Together with (\ref{3.11}), (\ref{3.12}) and (\ref{3.13}), we obtain $f_2(\rho_1)=2g_3(\rho_1)<2g_3(n+\frac{2d}{n}-3)=2p_1(d)\leq 2p_1(\frac{n-2}{2})=2p_2(n)\leq2p_2(14)<0,$ which implies $\rho_1=\rho(H_1)<\rho_2=\rho(H_2)$. Then for $k\in\{1,2\}$, $\rho(G_k)\leq\rho(H_2)$ with equality if and only if $k=2$ and $G_k\cong H_2$.
	
	{\noindent\textbf{Subcase 1.1.2.}} $k\geq 3$.
	
	In view of Lemma \ref{f*}, $\rho(G_k)$ is the largest root of $F_k(x)=0$, where $F_k(x)$ defined in Lemma \ref{f*}. By plugging $\rho_2=\rho(H_2)$ into $x$ of $F_k(x)-f_2(x)$, we obtain \begin{equation}\label{3.14}
		F_k(\rho_2)=F_k(\rho_2)-f_2(\rho_2)=2g_4(\rho_2),
	\end{equation} where $g_4(x)=(k-2)x^2+(k-ks-10+4d)x-2k^2s+kns-ks^2-ks+8d^2-12d-4dn+8n-8$.

	By $d\geq \frac{n-2}{2}$ and $n\geq s+2k+1$, we admit $\frac{ks-k+10-4d}{2(k-2)}<2k+s-2<n+\frac dn-3<\rho_2$ and then \begin{equation}\label{3.15}
		g_4(\rho_2)>g_4(n+\frac dn-3)=p_3(s),
	\end{equation} where $p_3(x)=-kx^2+(2k-2k^2-\frac{dk}{n})x+4-28d+8d^2+6k+10n-2n^2+\frac{2d}{n}-5kn+kn^2-\frac{2d^2}{n^2}+2dk+\frac{4d^2}{n}-\frac{5dk}{n}+\frac{d^2k}{n^2}$.
	
	Since $\frac{2k-2k^2-\frac{dk}{n}}{2k}<0<s\leq n-2k-1$, $p_3(s)\geq p_3(n-2k-1)= p_4(k),$ where $p_4(x)=(2n-6+\frac{2d}{n})x^2+(3+d-n-\frac{4d}{n}+\frac{d^2}{n^2})x+4-28d+8d^2+10n-2n^2-\frac{2d^2}{n^2}+\frac{2d}{n}+\frac{4d^2}{n}$. By $\frac{3+d-n-\frac{4d}{n}+\frac{d^2}{n^2}}{-2(2n-6+\frac{2d}{n})}<3\leq k,$ we have $p_4(x)$ is increasing in the interval $[3,+\infty)$ and then $p_4(k)\geq p_4(3)=(8+\frac 4n+\frac{1}{n^2})d^2+(\frac 8n-25)d-41+25n-2n^2$. 
	
	Let $p_5(x)=(8+\frac 4n+\frac{1}{n^2})x^2+(\frac 8n-25)x-41+25n-2n^2$. Then $p_5(d)\geq p_5(\frac{n-2}{2})=\frac{11n}{2}-\frac 5n+\frac{1}{n^2}-\frac{31}{4}>0$ by $\frac{25-\frac8n}{2(8+\frac 4n+\frac{1}{n^2})}<\frac{n-2}{2}\leq d$. 
	
	Combining (\ref{3.14}) and (\ref{3.15}), we have $F_k(\rho_2)=2g_4(\rho_2)>2g_4(n+\frac dn-3)=2p_3(s)\geq2p_3(n-2k-1)=2p_4(k)\geq2p_4(3)=2p_5(d)\geq2p_5(\frac{n-2}{2})>0.$ By Lemma \ref{f*}, we have $F_k(x)\geq F_k(\rho_2)>0$ when $x\geq\rho_2$, which implies $\rho(G_k)<\rho_2=\rho(H_2)$.
	
	 {\noindent\textbf{Subcase 1.2.}} $n=2k+s$.

	In this subcase, we have $G_k=K_s\vee2kK_1=K_{n-2k}\vee2kK_1$ and the quotient matrix of adjacent matrix $A(G_k)$ by virtue of the partition $\{V(K_{n-2k}),V(2kK_1)\}$ can be expressed as $\begin{pmatrix}
		n-2k-1&2k\\
		n-2k&0\\
	\end{pmatrix}.$ The corresponding characteristic polynomial is $\theta_k(x)=x^2-(n-2k-1)x-2k(n-2k).$ By Lemma \ref{Spectra}, $\rho(G_k)$ is the largest root of $\theta_k(x)=0$. Then we have \begin{equation}\label{3.17}
	\begin{aligned}
		\rho_2\theta_k(\rho_2)=\rho_2\theta_k(\rho_2)-f_2(\rho_2)=g_5(\rho_2),
	\end{aligned}
	\end{equation}where $g_5(x)=(2k-5)x^2+(4k^2-2kn-21+n+8d)x+16d^2-24d-8dn+16n-16$. 
	
	By $2d+1\leq n=2k+s\leq2k+k(d-2)=kd$, we have $k\geq 3$. Since $\frac{-(4k^2-2kn-21+n+8d)}{2(2k-5)}<n-3<n+\frac dn-3<\rho_2$, we have $g_5(x)$ is increasing in the interval $[n+\frac dn-3,+\infty)$ and  $
		g_5(\rho_2)>g_5(n+\frac dn-3)=(\frac{4d}{n}+4n-12)k^2+(18+2d-6n-\frac{12d}{n}+\frac{2d^2}{n^2})k+2-57d+16d^2+22n-4n^2+\frac{9d}{n}-\frac{5d^2}{n^2}+\frac{8d^2}{n}.$ 
		
		Let $p_6(x)=(\frac{4d}{n}+4n-12)x^2+(18+2d-6n-\frac{12d}{n}+\frac{2d^2}{n^2})x+2-57d+16d^2+22n-4n^2+\frac{9d}{n}-\frac{5d^2}{n^2}+\frac{8d^2}{n}$. By $\frac{-(18+2d-6n-\frac{12d}{n}+\frac{2d^2}{n^2})}{2(\frac{4d}{n}+4n-12)}<3\leq k$, $p_6(x)$ is increasing in the interval $[3,+\infty)$, which implies $p_6(k)\geq p_6(3)=(16+\frac{1}{n^2}+\frac{8}{n})d^2+(\frac 9n-51)d-52+40n-4n^2$. Let $p_7(x)=(16+\frac{1}{n^2}+\frac{8}{n})x^2+(\frac 9n-51)x-52+40n-4n^2$. By $\frac{-(\frac 9n-51)}{2(16+\frac{1}{n^2}+\frac{8}{n})}<4\leq\frac{n-2}{2}\leq d$, $p_7(d)\geq p_7(\frac{n-2}{2})=\frac n2-\frac 2n+\frac{1}{n^2}+\frac{47}{4}>0.$

	It can be inferred from (\ref{3.17}) and above arguments that $\rho_2\theta_k(\rho_2)=g_5(\rho_2)>g_5(n+\frac dn -3)=p_6(k)\geq p_6(3)=p_7(d)\geq p_7(\frac{n-2}{2})>0$ and so $\theta_k(\rho_2)>0$. Since $\frac{n-2k-1}{2}<n-3<n+\frac dn-3<\rho_2$, $\theta_k(x)$ is increasing in the interval $[\rho_2,+\infty)$. When $x\geq\rho_2$, $\theta_k(x)\geq \theta_k(\rho_2)>0$, which implies $\rho(G_k)<\rho_2=\rho(H_2)$.
	
	According to Subcase 1.1 and Subcase 1.2, we obtain $\rho(G)\leq\rho(G_k)\leq\rho(H_2)$ by (\ref{f3}). By the assumption $\rho(G)\geq\rho(H_2)$ in (\ref{th 1.13}), we have $\rho(G)=\rho(H_2)$, then $k=2$ and $G\cong G_k\cong H_2$, which is a contradiction with $G\ncong H_l$.

	{\noindent\textbf{Case 2.}} $\frac{n}{2}\leq d\leq n-3$.
	
	In this case, we verify $\rho(G_k)\leq\rho(R)$, where $R=K_{n-4}\vee4K_1$.  
	
	{\noindent\textbf{Subcase 2.1.}} $n\geq s+2k+1$.
	
	{\noindent\textbf{Subcase 2.1.1.}} $k=1$.
	
	 By Subcase 1.2, $\rho(R)$ is the largest root of $\theta_2(x)=0$. By a direct calculation, \begin{equation}\label{3.19}
		\rho_1\theta_2(\rho_1)=\rho_1\theta_2(\rho_1)-f_1(\rho_1)=g_6(\rho_1),
	\end{equation} where $g_6(x)=x^2-(3n-9-2d)x+2d^2-2d-2dn+4n-4.$ Since $\frac{3n-9-2d}{2}\leq \frac{3n-9-n}{2}<n-3<\rho_1<n+\frac{2d}{n}-3$, we have \begin{equation}\label{3.20}
	\begin{aligned}
		g_6(\rho_1)<g_6(n+\frac{2d}{n}-3)=&(2+\frac 4n+\frac{4}{n^2})d^2+(\frac 6n-10)d-22+16n-2n^2.
	\end{aligned}
	\end{equation}
	
	 Let $p_8(x)=(2+\frac 4n+\frac{4}{n^2})x^2+(\frac 6n-10)x-22+16n-2n^2$. By $\frac{-(\frac 6n-10)}{2(2+\frac 4n+\frac{4}{n^2})}=\frac{5n^2-3n}{2n^2+4n+4}< 4<\frac n2\leq d \leq n-3$, we have $p_8(d)\leq p_8(n-3)=-2n+12+\frac{36}{n^2}-\frac6n<0.$ 
	 
	 Combining (\ref{3.19}) and (\ref{3.20}), we admit $\rho_1\theta_2(\rho_1)=g_6(\rho_1)<g_6(n+\frac{2d}{n}-3)=p_8(d)\leq p_8(n-3)<0,$ which implies $\rho_1=\rho(H_1)<\rho(R)$. Since $G_1$ is a spanning graph of $H_1$, $\rho(G_1)\leq\rho(H_1)<\rho(R)$.

	{\noindent\textbf{Subcase 2.1.2.}} $k\geq 2$.
	
	Recall that $\rho(G_k)$ is the largest root of $F_k(x)=0$. By a direct computation, \begin{equation}\label{f5}
		F_k(\rho(R))=F_k(\rho(R))-\rho \theta_2(\rho(R))=g_7(\rho(R)),
	\end{equation}
			where $g_7(x)=(2k-3)x^2+(3n-15+2k-2ks)x-4k^2s+2kns-2ks^2-2ks$. Since $\theta_2(n-2)=-n+10<0$ and $n>10$, we have $\rho(R)>n-2$. By $s\leq n-2k-1$ and $\frac{-(3n-15+2k-2ks)}{2(2k-3)}<n-2<\rho(R)$, we obtain \begin{equation}\label{f6}
				g_7(\rho(R))>g_7(n-2)=-2ks^2+(2k-4k^2)s+18+4k-9n-6kn+2kn^2.
			\end{equation}
	
	Let $p_9(x)=-2kx^2+(2k-4k^2)x+18+4k-9n-6kn+2kn^2$. By $\frac{2k-4k^2}{4k}<0<s\leq n-2k-1$, $p_9(s)\geq p_9(n-2k-1)=(4n-8)k^2-9n+18\geq 4(4n-8)-9n+18=7n-14>0$. 
	
	Combining (\ref{f5}) and (\ref{f6}), we obtain $F_k(\rho(R))=g_7(\rho(R))>g_7(n-2)=p_9(s)>0.$ By Lemma \ref{f*}, we have $F_k(x)\geq F_k(\rho(R))>0$ when $x\geq\rho(R)$, which implies $\rho(G_k)<\rho(R)$. 
	
	{\noindent\textbf{Subcase 2.2.}} $n= s+2k$.
	
	In this subcase, we admit $G_k=K_s\vee2kK_1=K_{n-2k}\vee2kK_1$ and $d+3\leq n=2k+s\leq2k+k(d-2)=kd$. Then $k\geq 2$ and $G_k$ is a spanning subgraph of $R$. By Lemma \ref{rq1}, $\rho(G_k)\leq\rho(R)$ with equality if and only if $G_k\cong R$.
	
	According to Subcase 2.1 and Subcase 2.2, we obtain $\rho(G)\leq\rho(G_k)\leq\rho(R)$ by (\ref{f3}). By the assumption $\rho(G)\geq\rho(R)$ in (\ref{th 1.13}), we have $\rho(G)=\rho(G_k)=\rho(R)$, then $G\cong G_k\cong R$, which implies $G=K_{n-4}\vee4K_1$ has a spanning tree with leaf distance at least $d$, a contradiction.
	
	{\noindent\textbf{Case 3.}} $n-2\leq d\leq n-1$.
	
	In this case, we verify $\rho(G_k)\leq\rho(H_1)$. 
	
	{\noindent\textbf{Subcase 3.1.}} $n\geq s+2k+1$.
	
	For $k=1$, we have $\rho(G_1)\leq\rho(H_1)$ with equality if and only if $G_1\cong H_1$ since $G_1$ is a spanning subgraph of $H_1$ and Lemma \ref{rq1}.
	
	Now we prove the subcase when $k\geq 2$. Since $f_1(n-2)=-2d^2+6d+n^2-3n-2\leq-2(n-2)^2+6(n-2)+n^2-3n-2=-n^2+11n-22<0$ by $d\geq n-2$, we obtain $\rho_1>n-2$. By plugging $\rho_1=\rho(H_1)$ into $x$ of $F_k(x)-f_1(x)$, we obtain \begin{equation}\label{3.23}
		F_k(\rho_1)=F_k(\rho_1)-f_1(\rho_1)=2g_8(\rho_1),
	\end{equation} where $g_8(x)=(k-1)x^2+(k-ks+d-3)x-2k^2s+kns-ks^2-ks+d^2-d-dn+2n-2$. 
	
	By $\frac{-(k-ks+d-3)}{2(k-1)}<n-2<\rho_1$, we derive $g_8(\rho_1)>g_8(n-2)= p_{10}(s)$, where $p_{10}(x)=-kx^2+(k-2k^2)x-3d+d^2+2k+3n-n^2-3kn+kn^2$. By $\frac{k-2k^2}{2k}<0<s\leq n-2k-1$, $p_{10}(s)\geq p_{10}(n-2k-1)=(2n-4)k^2-3d+d^2+3n-n^2\geq4(2n-4)-3d+d^2+3n-n^2= p_{11}(d)\geq p_{11}(n-2)=4n-6>0$, where $p_{11}(x)=x^2-3x-n^2+11n-16$. 
	
	In view of (\ref{3.23}), $F_k(\rho_1)=2g_8(\rho_1)>2g_8(n-2)=2p_{10}(s)\geq2p_{10}(n-2k-1)=2p_{11}(d)\geq2p_{11}(n-2)>0.$ By Lemma \ref{f*}, we have $F_k(x)\geq F_k(\rho_1)>0$ when $x\geq\rho_1$, which implies $\rho(G_k)<\rho_1=\rho(H_1)$. 

	{\noindent\textbf{Subcase 3.2.}} $n=2k+s$.
	
	In this subcase, $G_k=K_{n-2k}\vee2kK_1$. By $d+1\leq n=2k+s\leq 2k+k(d-2)=kd$, we have $k\geq 2$. Thus $G_k$ is a spanning subgraph of $R$ and $\rho(G_k)\leq\rho(R)$. In what follows, we verify $\rho(R)<\rho(H_1)$. Recall that $g_6(\rho_1)=\rho_1\theta_2(\rho_1)=\rho_1\theta_2(\rho_1)-f_1(\rho_1)=\rho_1^2-(3n-9-2d)\rho_1+2d^2-2d-2dn+4n-4$. 
	
	By $\frac{3n-9-2d}{2}\leq \frac{n-5}{2}<n-2<\rho_1$, $g_6(\rho_1)>g_6(n-2)=2d^2-6d-18+15n-2n^2=p_{12}(d)\geq p_{12}(n-2)=2+n>0$, which implies $\rho_1\theta_2(\rho_1)=g_6(\rho_1)>g_6(n-2)=p_{12}(d)\geq p_{12}(n-2)>0$. Since $\theta_2(x)$ is increasing in the $[\rho_1,+\infty)$, $\theta_2(x)\geq \theta_2(\rho_1)>0$ when $x\geq\rho_1$, then $\rho(R)<\rho_1=\rho(H_1)$. Hence $\rho(G_k)\leq\rho(R)<\rho(H_1)$.
	
	According to Subcase 3.1 and Subcase 3.2, we obtain $\rho(G)\leq\rho(G_k)\leq\rho(H_1)$ by (\ref{f3}). By the assumption $\rho(G)\geq\rho(H_1)$ in (\ref{th 1.13}), we have $\rho(G)=\rho(H_1)$, then $\rho(G)=\rho(G_k)=\rho(H_1)$ and $G\cong G_1\cong H_1$, which contradicts with $G\cong H_l$.
	
	Combining the above arguments, $G$ has a spanning tree with leaf distance at least $d$.
	
	{\hfill $\blacksquare$ \par}
	
\section{The existence of spanning trees with leaf degree at most $k$ with respect to the minimum degree}

\hspace{1.5em}Kaneko \cite{AK1} introduced the concept of leaf degree of a spanning tree and provided the following criterion for a connected graph to have a spanning tree with leaf degree at most $k$.

\begin{theorem}{\rm(\!\!\cite{AK1})}\label{t1}
	Let $G$ be a connected simple graph and $k\geq 1$ be an integer. Then $G$ has a spanning tree with leaf degree at most $k$ if and only if $i(G-S)<(k+1)|S|$ for any nonempty subset $ S \subseteq V(G)$, where $i(G-S)$ is the number of isolated vertices in $G-S$.
\end{theorem}

Let $G$ be a graph of order $n$ with the minimum degree $\delta(G)\geq t$. Assume that $G$ contains no spanning tree with leaf degree at most $k$. Then there exists a nonempty subset $S$ such that $i(G-S)\geq(k+1)|S|$ by Theorem \ref{t1}. Clearly, $|S|\geq \delta(G)\geq t$ and thus $n\geq (k+2)|S|\geq (k+2)t$. Hence if $n<(k+2)t$, $G$ contains a spanning tree with leaf degree at most $k$.

\begin{corollary}
	Let $G$ be a graph of order $n$ with the minimum degree $\delta(G)\geq t$. If $n<(k+2)t$, then $G$ contains a spanning tree with leaf degree at most $k$.
\end{corollary}

In what follows, we only consider $n\geq (k+2)t$. 

By using Theorem \ref{t1} and typical spectral techniques, many researchers have studied the existence of spanning trees with leaf degree at most $k$ in connected graphs. Ao, Liu and Yuan \cite{GA} as well as Zhou, Sun and Liu \cite{DQ} established the spectral conditions for the existence of a spanning tree with leaf degree at most $k$. Combining the minimum degree of graphs, Ao, Liu, Yuan, Ng and Cheng \cite{GYA} also presented the following adjacent spectral result to ensure that a connected graph has a spanning tree with leaf degree at most $k$.

\begin{theorem}{\rm(\!\!\cite{GYA})}\label{t2}
	Let $G$ be a connected graph of order $n\geq 3(k+2)\delta+2$ with minimum degree $\delta$, where $k\geq1$ is an integer. If $\rho(G)\geq\rho(K_\delta\vee(K_{n-k\delta-2\delta}\cup(k\delta+\delta)K_1))$, then $G$ has a spanning tree with leaf degree at most $k$ unless $G\cong K_\delta\vee(K_{n-k\delta-2\delta}\cup(k\delta+\delta)K_1)$.
\end{theorem}

In fact, it is not easy to verify whether a graph contains a spanning tree with a leaf degree of at most $k$ by using Theorem \ref{t1}, because we need demonstrate $i(G-S)<(k+1)|S|$ for all nonempty subsets $S\subseteq V(G)$. Motivated by \cite{AK1,GYA}, in this paper, we study the sufficient conditions to ensure that a graph $G$ with the minimum degree $\delta(G)$ has a spanning tree with leaf degree at most $k$ in terms of its lower bound of size and spectral radius of $G$, and obtain Theorems \ref{ta}, \ref{tb}, \ref{tc} as follows.

\begin{theorem}\label{ta}
	Let $G$ be a connected graph of order $n\geq (k+2)t$ with the minimum degree $\delta(G)$, where $ t\leq\delta(G)$ and $k$ are positive integers. 
	
	\noindent{\rm (i)} For $n\geq(k+4)t+1+\frac{2t+2}{k+1}$, if $|E(G)|\geq |E(K_t\vee(K_{n-(k+2)t}\cup(k+1)tK_1))|$, then $G$ has a spanning tree with leaf degree at most $k$ unless $G\cong K_t\vee(K_{n-(k+2)t}\cup(k+1)tK_1)$.
	
	\noindent{\rm (ii)} For $n\leq(k+4)t+1+\frac{2t+1}{k+1}$, if $|E(G)|\geq	\max\{|E(K_t\vee(K_{n-(k+2)t}\cup(k+1)tK_1))|,$ $|E(K_{\lfloor{\frac{n}{k+2}}\rfloor}\vee(K_{n-(k+2)\lfloor{\frac{n}{k+2}}\rfloor}\cup(k+1)\lfloor{\frac{n}{k+2}}\rfloor K_1))|\}$, then $G$ has a spanning tree with leaf degree at most $k$ unless $G\cong K_t\vee(K_{n-(k+2)t}\cup(k+1)tK_1)$ or $K_{\lfloor{\frac{n}{k+2}}\rfloor}\vee(K_{n-(k+2)\lfloor{\frac{n}{k+2}}\rfloor}\cup(k+1)\lfloor{\frac{n}{k+2}}\rfloor K_1)$.
\end{theorem}

Using different methods, we obtain Theorem \ref{tb} which is an improvement of Theorem \ref{t2}.

\begin{theorem}\label{tb}
	Let $G$ be a connected graph of order $n\geq 2(k+2)t$, where $t\leq\delta(G)$ and $k$ are positive integers.
	
	\noindent{\rm (i)} If $(k,n)\neq(1,6t)$ and $\rho(G)\geq\rho(K_t\vee(K_{n-(k+2)t}\cup(k+1)tK_1))$, then $G$ has a spanning tree with leaf degree at most $k$ unless $G\cong K_t\vee(K_{n-(k+2)t}\cup(k+1)tK_1)$.
	
	\noindent{\rm (ii)} If $(k,n)=(1,6t)$, $t\geq 3$ and $\rho(G)\geq\rho(K_t\vee(K_{3t}\cup2tK_1)$, then $G$ has a spanning tree with leaf degree at most $k$ unless $G\cong K_t\vee(K_{3t}\cup2tK_1)$.
	
	\noindent{\rm (iii)} If $(k,n)=(1,6t)$, $t\in\{1,2\}$ and $\rho(G)\geq \rho(K_{2t}\vee4tK_1)$, then $G$ has a spanning tree with leaf degree at most $k$ unless $G\cong K_{2t}\vee4tK_1$.	
\end{theorem}

\begin{theorem}\label{tc}
	Let $G$ be a connected graph of order $n\geq 2(k+3)t+1$, where $t\leq\delta(G)$ and $k$ are positive integers. If $q(G)\geq q(K_t\vee(K_{n-(k+2)t}\cup(k+1)tK_1)$, then $G$ has a spanning tree with leaf degree at most $k$ unless $G\cong K_t\vee(K_{n-(k+2)t}\cup(k+1)tK_1$.
\end{theorem}	
	
In what follows, we give the proofs of Theorems \ref{ta}, \ref{tb} and \ref{tc}, respectively. Firstly, we present the following proposition.

\begin{proposition}\label{leaf degree}
	Let $n$, $k$, $l$ be integers with $n\geq (k+2)l$. Then $K_l\vee(K_{n-(k+2)l}\cup(k+1)lK_1)$ has no spanning tree with leaf degree at most $k$. In particular, if $n\geq (k+2)l+1$, then
	
	\noindent{\rm (i)} $\rho(K_l\vee(K_{n-(k+2)l}\cup(k+1)lK_1))$ is the largest root of $x^{3}-\left(n-l k-l-2\right) x^{2}-\left(l^{2} k-l k+l^{2}-l+n-1\right) x-k^{2} l^{3}-3 k \,l^{3}+k \,l^{2} n-l^{2} k-2 l^{3}+l^{2} n-l^{2}=0$.
	
	\noindent{\rm (ii)} $\rho(K_l\vee(k+1)lK_1)$ is the largest root of $x^{2}-\left(l-1\right) x-l^{2} \left(k+1\right)=0$.
	
	\noindent{\rm (iii)} $q(K_l\vee(K_{n-(k+2)l}\cup(k+1)lK_1))$ is the largest root of $x^{3}-\left(3 n-2 l k-l-4\right) x^{2}-\left(4 l^{2} k+2 k l n-4 l k+4 l^{2}-l n-2 n^{2}+6 n-4\right) x -2 k^{2} l^{3}-4 k \,l^{3}+4 k \,l^{2} n-6 l^{2} k-2 l^{3}+4 l^{2} n-2 l \,n^{2}-6 l^{2}+6 l n-4 l=0$.
	
	\noindent{\rm (iv)} $q(K_l\vee(k+1)lK_1)$ is the largest root of $x^{2}-\left(k l+4 l-2\right) x+2 l^{2}-2 l=0$.
	
	\noindent{\rm (v)} $\rho(K_{2l}\vee 4lK_1)>\rho(K_l\vee(K_{3l}\cup 2lK_1))$ for $l\in\{1,2\}$.
	
	\noindent{\rm (vi)} $\rho(K_{2l}\vee 4lK_1)<\rho(K_l\vee(K_{3l}\cup 2lK_1))$ for $l\geq 3$.
\end{proposition} 

\begin{proof}
	Let $G=K_l\vee(K_{n-(k+2)l}\cup(k+1)lK_1)$ and $S=V(K_l)$. Since $i(G-S)\geq (k+1)l=(k+1)|S|$, $G$ has no spanning tree with leaf degree at most $k$ by Theorem \ref{t1}.
	
	Considering the partition $V(K_l\vee(K_{n-(k+2)l}\cup(k+1)lK_1))=V(K_l)\cup V(K_{n-(k+2)l})\cup V((k+1)lK_1)$ and $V(K_l\vee(k+1)lK_1)=V(K_{l})\cup V((k+1)lK_1)$, the corresponding quotient matrices of $A(K_l\vee(K_{n-(k+2)l}\cup(k+1)lK_1))$, $A(K_l\vee(k+1)lK_1)$, $Q(K_l\vee(K_{n-(k+2)l}\cup(k+1)lK_1))$ and $Q(K_l\vee(k+1)lK_1)$ equal
	$$B_1=\begin{pmatrix}
		l-1 & n-(k+2)l& (k+1)l\\
		l & n-(k+2)l-1 & 0\\
		l & 0 & 0\\
	\end{pmatrix}\text{, } B_2=\begin{pmatrix}
		l-1 & (k+1)l\\
		l & 0 \\
	\end{pmatrix}, $$$$B_3=\begin{pmatrix}
		n+l-2 & n-(k+2)l& (k+1)l\\
		l & 2n-2kl-3l-2 & 0\\
		l & 0 & l\\
	\end{pmatrix},B_4=\begin{pmatrix}
		kl+3l-2 & (k+1)l\\
		l & l \\
	\end{pmatrix}.$$
	
	By direct computation, we derive the characteristic polynomials of the matrices $B_1$, $B_2$ and $B_3$ as $h_1(x)=x^{3}-\left(n-l k-l-2\right) x^{2}-\left(l^{2} k-l k+l^{2}-l+n-1\right) x-k^{2} l^{3}-3 k \,l^{3}+k \,l^{2} n-l^{2} k-2 l^{3}+l^{2} n-l^{2}$, $h_2(x)=x^{2}-\left(l-1\right) x-l^{2} \left(k+1\right)$, $h_3(x)=x^{3}-\left(3 n-2 l k-l-4\right) x^{2}-(4 l^{2} k+2 k l n-4 l k+4 l^{2}-l n-2 n^{2}+6 n-4) x -2 k^{2} l^{3}-4 k \,l^{3}+4 k \,l^{2} n-6 l^{2} k-2 l^{3}+4 l^{2} n-2 l \,n^{2}-6 l^{2}+6 l n-4 l$ and $h_4(x)=x^{2}-\left(k l+4 l-2\right) x+2 l^{2}-2 l$, respectively. Since the partition $V(K_l\vee(K_{n-(k+2)l}\cup(k+1)lK_1))=V(K_l)\cup V(K_{n-(k+2)l})\cup V((k+1)lK_1)$ and $V(K_l\vee(k+1)lK_1)=V(K_{l})\cup V((k+1)lK_1)$ are equitable, $\rho(K_l\vee(K_{n-(k+2)l}\cup(k+1)lK_1))$, $\rho(K_l\vee(k+1)lK_1)$, $q(K_l\vee(K_{n-(k+2)l}\cup(k+1)lK_1))$ and $q(K_l\vee(k+1)lK_1)$ are the largest root of $h_1(x)$, $h_2(x)$, $h_3(x)$ and $h_4(x)$ by Lemma \ref{Spectra}, respectively. So, (i), (ii), (iii) and (iv) are proved.
	
	By (i), (ii) and $(k,n)=(1,6l)$, $\rho(K_l\vee(K_{3l}\cup 2lK_1))$ and $\rho(K_{2l}\vee 4lK_1)$ are the largest root of $h_5(x)=x^{3}-\left(4 l-2\right) x^{2}-\left(2 l^{2}+4 l-1\right) x+6 l^{3}-2 l^{2}$ and $h_6(x)=x^{2}+\left(-2 l+1\right) x-8 l^{2}$. Then $\rho(K_{2l}\vee 4lK_1)=l-\frac12+\sqrt{9l^2-l+\frac14}$ and thus $h_5(\rho(K_{2l}\vee 4lK_1))=l^{2} \left(\sqrt{36 l^{2}-4 l+1}-8 l+5\right)>0$ for $l\in\{1,2\}$. Since $h_5'(x)=3 x^{2}-2 \left(4 l-2\right) x-2 l^{2}-4 l+1$ and $h_5'(x)\geq h_5'(\rho(K_{2l}\vee 4lK_1))>h_5'(4l-1)=14l^2-4l>0$ by $x\geq \rho(K_{2l}\vee 4lK_1)$ and $\frac{4l-2}{3}<4l-1<\rho(K_{2l}\vee 4lK_1)$, $h_5(x)$ is increasing in the interval $[\rho(K_{2l}\vee 4lK_1),+\infty)$, which implies when $x\geq \rho(K_{2l}\vee 4lK_1)$, $h_5(x)\geq h_5(\rho(K_{2l}\vee 4lK_1))>0$. Hence $\rho(K_{2l}\vee 4lK_1)>\rho(K_l\vee(K_{3l}\cup 2lK_1))$.
	
	It is easy to check that $h_5(\rho(K_{2l}\vee 4lK_1))=l^{2} \left(\sqrt{36 l^{2}-4 l+1}-8 l+5\right)<0$ for $l\geq 3$. Then $\rho(K_{2l}\vee 4lK_1)<\rho(K_l\vee(K_{3l}\cup 2lK_1))$.
\end{proof}

\subsection{Proof of Theorem \ref{ta}} 
\hspace{1.5em}Suppose that $G$ has no spanning tree with leaf degree at most $k$. Then there exists a nonempty subset $S\subseteq V(G)$ such that $i(G-S)\geq (k+1)|S|$ by Theorem \ref{t1}.

Without loss of generality, we choose such connected graph $R$ of the vertex set $V(R)=V(G)$ so that its size is as large as possible and $\delta(R)\geq t$. Then $G$ is a spanning subgraph of $R$ and $|E(G)|\leq |E(R)|$. According to the choice of $R$, we see that the induced subgraph $R[S]$ and each connected component of $R-S$ are complete graphs, and then $R= R[S]\vee(R-S)$. 

Firstly, we know that there is at most one non-trivial connected component in $R-S$. Otherwise, we can obtain a new graph ${R}'$ by adding edges among all non-trivial connected components to derive a bigger non-trivial connected component. Obviously, $R$ is a proper spanning subgraph of ${R}'$, which is a contradiction to the choice of $R$. Let $i(R-S)=i$ and $|S|=s$ for short. Then $i\geq (k+1)s$ by $i(R-S)\geq (k+1)|S|$.

We complete the proof by considering the following two cases.

{\noindent\textbf{Case 1.}} $R-S$ has exactly one non-trivial connected component. 

In this case, let the unique non-trivial of $R$ have $n_1$ vertices. connected component of $R-S$ and $V(R_1)=n_1\geq2$. Now, we show $i=(k+1)s$. 

If $i\geq (k+1)s+1$, then we take a new graph ${R}''$ obtained from $R$ by joining each vertex of $R_1$ with one vertex in $V(R-S)\backslash V(R_1)$. Then we have $|E({R}'')|=|E(R)|+n_1>|E(R)|,\text{ } \delta({R}'')\geq\delta(R)\geq t ,\text{ } i({R}''-S)=i-1\geq (k+1)s=(k+1)|S|,$ which is a contradiction with the choice of $R$. Thus $i=(k+1)s$ by $i\geq (k+1)s$, $R=K_s\vee(K_{n-(k+2)s}\cup (k+1)sK_1)$ and $\delta(R)=s\geq t$. Clearly, we have $n=(k+2)s+n_1\geq (k+2)s+2\geq (k+2)t+2$. Let $H=K_t\vee(K_{n-(k+2)t}\cup(k+1)tK_1$ and $\widehat{H}=K_\phi\vee(K_{n-(k+2)\phi}\cup(k+1)\phi K_1)$, where $\phi=\lfloor{\frac{n}{k+2}}\rfloor$.

For $n\geq (k+4)t+1+\frac{2t+2}{k+1}$, we have
$$\begin{aligned}
	&|E(H)|-|E(R)|\\=&\binom{n-(k+1)t}{2}+(k+1)t^2-\binom{n-(k+1)s}{2}-(k+1)s^2\\=&\frac12 \left(s-t\right) \left(k+1\right) \left(2 n-(k+3)s-(k+3)t-1\right)\\\geq&\frac12 \left(s-t\right) \left(k+1\right) \left(2 n-(k+3)\frac{n-2}{k+2}-(k+3)t-1\right)\\\geq&\frac{(s-t)(k+1)}{2(k+2)}((k+1)((k+4)t+1+\frac{2t+2}{k+1})-k^2t-(5t-1)k-6t+4)\\=&\frac{(s-t)(k+1)}{2(k+2)}(2k+7)\geq 0.
\end{aligned}$$

By the assumption of (i), $|E(H)|\leq|E(G)|\leq|E(R)|\leq|E(H)|$ and then $s=t$. Hence $G= R\cong H$ and $H$ has no spanning tree with degree at most $k$ by Proposition \ref{leaf degree}.

For $n\leq (k+4)t+1+\frac{2t+1}{k+1}$ and $s=t$, we have $G=R\cong H$.

For $n\leq (k+4)t+1+\frac{2t+1}{k+1}$ and $s\geq t+1$, we have $$|E(\widehat{H})|-|E(R)|=\frac12(\phi-s)(k+1)((k+3)s+(k+3)\phi+1-2n).$$
Let $n=(k+2)l+r$, where $0\leq r\leq k+1$ by the Euclidean division. Then $\phi=\lfloor{\frac{n}{k+2}}\rfloor=\lfloor{l+\frac{r}{k+2}}\rfloor=l$ and thus $n=(k+2)\phi+r$ i.e. $\phi=\frac{n-r}{k+2}$. Hence $(k+3)s+(k+3)\phi+1-2n=(k+3)s+(k+3)\frac{n-r}{k+2}+1-2n=(k+3)s-\frac{(k+1)n}{k+2}-\frac{(k+3)r}{k+2}+1\geq (k+3)s-\frac{(k+1)((k+4)t+1+\frac{2t+1}{k+1})}{k+2}-\frac{(k+3)r}{k+2}+1=(k+3)s-(k+3)t-\frac{(k+3)r}{k+2}\geq (k+3)(s-t)-\frac{(k+3)(k+1)}{k+2}=(k+3)(s-t)-k-2+\frac{1}{k+2}\geq 1+\frac{1}{k+2}>0$. Due to $\phi=\lfloor{\frac{n}{k+2}}\rfloor\geq \lfloor{\frac{(k+2)s+2}{k+2}}\rfloor=s$ and the assumption of (ii), $|E(\widehat{H})|\leq|E(G)|\leq|E(R)|\leq|E(\widehat{H})|$ and then $s=\phi$. Hence $G= R\cong \widehat{H}$ and $\widehat{H}$ has no spanning tree with degree at most $k$ by Proposition \ref{leaf degree}.

{\noindent\textbf{Case 2.}} $R-S$ has no non-trivial connected component.

In this case, we show $i\leq (k+1)s+1$. If $i\geq (k+1)s+2$, then we take a new graph ${R}'''$ obtained from $R$ by adding an edge between two vertices in $V(R-S)$. Thus $i({R}'''-S)=i-2\geq (k+1)s=(k+1)|S|$, $\delta({R}''')\geq\delta(R)\geq t$, and $|E({R}''')|=|E(R)|+1$. This contradicts with the choice of $R$, which implies $i=(k+1)s$ or $i=(k+1)s+1$.

{\noindent\textbf{Subcase 2.1.}} $i=(k+1)s$.

In this subcase, $n=(k+2)s\geq(k+2)t$, $R=K_s\vee (k+1)sK_1$ and $|E(R)|=\binom{s}{2}+(k+1)s^2$.

For $n\geq (k+4)t+1+\frac{2t+2}{k+1}$, we deduce $$\begin{aligned}
	|E(H)|-|E(R)|=&\binom{(k+2)s-(k+1)t}{2}+(k+1)t^2-\binom{s}{2}-(k+1)s^2\\=&\frac12(k+1)(s-t)((k+1)s-(k+3)t-1)\\= &\frac12(k+1)(s-t)\left(\frac{(k+1)n}{k+2}-(k+3)t-1\right)\\\geq &\frac{(k+1)(s-t)}{2(k+2)}\geq 0.
\end{aligned}$$  By the assumption of (i), $|E(H)|\leq|E(G)|\leq|E(R)|\leq|E(H)|$ and then $s=t$. Hence $n\geq (k+4)t+1+\frac{2t+2}{k+1}=(k+4)s+1+\frac{2s+2}{k+1}>(k+2)s$, a contradiction with $n=(k+2)s$.

For $n\leq (k+4)t+1+\frac{2t+1}{k+1}$, we have $R=K_s\vee (k+1)sK_1=\widehat{H}$ by $\phi=\lfloor{\frac{n}{k+2}}\rfloor=s$, and $\widehat{H}$ has no spanning tree with degree at most $k$ by Proposition \ref{leaf degree}.

{\noindent\textbf{Subcase 2.2.}} $i=(k+1)s+1$.

In this subcase, $n=(k+2)s+1\geq (k+2)t+1$, $R=K_s\vee ((k+1)s+1)K_1$ and $|E(R)|=\binom{s}{2}+s((k+1)s+1)$. 

For $n\geq (k+4)t+1+\frac{2t+2}{k+1}$, we have $$\begin{aligned}|E(H)|-|E(R)|=&\frac12(k+1)(s-t)((k+1)s-(k+3)t+1)\\=&\frac12(k+1)(s-t)\left(\frac{(k+1)(n-1)}{k+2}-(k+3)t+1\right)\\\geq &\frac{(k+1)(s-t)(k+4)}{2(k+2)}\geq 0.
\end{aligned}$$ 

By the assumption of (i), $|E(H)|\leq|E(G)|\leq|E(R)|\leq|E(H)|$ and then $s=t$. Hence $n\geq (k+4)t+1+\frac{2t+2}{k+1}=(k+4)s+1+\frac{2s+2}{k+1}>(k+2)s+1$, a contradiction with $n=(k+2)s+1$.

For $n\leq (k+4)t+1+\frac{2t+1}{k+1}$, we have $R=K_s\vee ((k+1)s+1)K_1=\widehat{H}$ by $\phi=\lfloor{\frac{n}{k+2}}\rfloor=s$, and $\widehat{H}$ has no spanning tree with degree at most $k$ by Proposition \ref{leaf degree}.

In view of Case 1 and Case 2, the proof of Theorem \ref{ta} is complete.
{\hfill $\blacksquare$ \par}

\subsection{Proof of Theorem \ref{tb}} 
\hspace{1.5em}Let $H=K_t\vee(K_{n-(k+2)t}\cup (k+1)tK_1)$, $\widetilde{H}= K_{2t}\vee4tK_1$, and $G$ be a connected graph of order $n\geq 2(k+2)t$ with $\delta(G)\geq t\geq 1$ and $\rho(G)\geq\begin{cases}\rho(H),&\text{if $(k,n)\neq(1,6t)$}\\ \rho(H),&\text{if $(k,n)=(1,6t)$ and $t\geq 3$}\\ \rho(\widetilde{H}),&\text{if $(k,n)=(1,6t)$ and $t\in \{1,2\}$} \end{cases}$. Suppose to the contrary that $G$ contains no spanning tree with leaf degree at most $k$ and then it suffices to show that $G\cong \begin{cases}H,&\text{if $(k,n)\neq(1,6t)$}\\ H,&\text{if $(k,n)=(1,6t)$ and $t\geq 3$}\\ \widetilde{H},&\text{if $(k,n)=(1,6t)$ and $t\in \{1,2\}$} \end{cases}$ by Proposition \ref{leaf degree}. By applying Theorem \ref{t1}, there exists some nonempty subset $S\subseteq V(G)$ such that $i(G-S)\geq (k+1)|S|.$ For convenience, we take $|S|=s$, then $n\geq (k+2)s$ and $s\geq \delta(G)\geq t$. Then $G^*\cong K_s\vee(K_{n-(k+2)s}\cup (k+1)sK_1)$ is the graph with the maximum size such that $G$ is a spanning subgraph of $G^*$ and $i(G^*-S)\geq(k+1)|S|$. Therefore, $\rho(G^*)\geq \rho(G)\geq \begin{cases}\rho(H),&\text{if $(k,n)\neq(1,6t)$}\\ \rho(H),&\text{if $(k,n)=(1,6t)$ and $t\geq 3$}\\ \rho(\widetilde{H}),&\text{if $(k,n)=(1,6t)$ and $t\in \{1,2\}$} \end{cases}$ by Lemma \ref{rq1}.

Next we proceed to prove Theorem \ref{tb} by considering the following two cases.

{\noindent\textbf{Case 1.}} $s=t$.

In this case, $G^*\cong H$, then $\rho(G^*)=\rho(H)$. 

If $(k,n)\neq(1,6t)$ or $(k,n)=(1,6t)$ and $t\geq 3$, then $\rho(H)=\rho(G^*)\geq \rho(G)\geq \rho(H)$ and thus $G=G^*\cong H$.

If $(k,n)=(1,6t)$ and $t\in\{1,2\}$, then $H=K_t\vee(K_{3t}\cup 2tK_1)$. By Proposition \ref{leaf degree}, we have $\rho(\widetilde{H})>\rho(H)$ and then $\rho(G^*)\geq\rho(G)\geq\rho(\widetilde{H})>\rho(H)=\rho(G^*)$, a contradiction.

{\noindent\textbf{Case 2.}} $s\geq t+1$.

By $n\geq (k+2)s$, we now consider the following two subcases.

{\noindent\textbf{Subcase 2.1.}} $n\geq(k+2)s+1$.

By Proposition \ref{leaf degree}, $\rho_1=\rho(G^*)$, $\rho_2=\rho(H)$ are the largest root of $f_1(x)=0$, $f_2(x)=0$, respectively, where $f_1(x)=x^{3}-\left(n-k s-s-2\right) x^{2}-\left(k \,s^{2}-k s+s^{2}+n-s-1\right) x-k^{2} s^{3}-3 k \,s^{3}+k \,s^{2} n-k \,s^{2}-2 s^{3}+s^{2} n-s^{2}$ and $f_2(x)=x^{3}-\left(n-k t-t-2\right) x^{2}-(k \,t^{2}-k t+t^{2}+n-t-1) x-k^{2} t^{3}-3 k \,t^{3}+k \,t^{2} n-k \,t^{2}-2 t^{3}+t^{2} n-t^{2}$. In terms of Lemma \ref{rq1}, $\rho_2=\rho(H)>\rho(K_{n-(k+1)t}\cup(k+1)tK_1)=n-(k+1)t-1$, since $K_{n-(k+1)t}\cup(k+1)tK_1$ is a spanning subgraph of $H$.

By direct calculation, we derive \begin{equation}\label{equ3}
	\begin{aligned}
		f_1(\rho_2)=f_1(\rho_2)-f_2(\rho_2)=(s-t)(k+1)f_3(\rho_2),
	\end{aligned}
\end{equation} where $f_3(x)=x^{2}+\left(1-s-t\right) x-\left(k+2\right) (s^{2}+st+t^2)+\left(n-1\right) (s+t)$. By $n\geq (k+2)s\geq(k+1)(t+1)$ and $\frac{s+t-1}{2}<n-(k+1)t-1<\rho_2$, we have \begin{equation}\label{equ4}
	\begin{aligned}
		f_3(\rho_2)>& f_3(n-(k+1)t-1)\\=&-\left(k+2\right) s^{2}-s t+k^{2} t^{2}-2 k n t+2 k \,t^{2}+k t+n^{2}-2 t n-n+t.
	\end{aligned}
\end{equation}  

Let $f_4(x)=-\left(k+2\right) x^{2}- tx+k^{2} t^{2}-2 k n t+2 k \,t^{2}+k t+n^{2}-2 t n-n+t$. Clearly, $f_4(x)$ is decreasing in the interval $[t+1, \frac {n-1}{k+2}]$. By $s\in[t+1, \frac {n-1}{k+2}]$, we have \begin{equation}\label{equ5}
	f_4(s)\geq f_4(\frac{n-1}{k+2})=\frac{1}{k+2}f_5(n),
\end{equation}
where $f_5(x)=\left(k+1\right) x^{2}-\left(2 k^{2} t+6 k t+k+5 t\right) x+k^{3} t^{2}+4 k^{2} t^{2}+k^{2} t+4 k \,t^{2}+3 k t+3 t-1$. Since $\frac{2 k^{2} t+6 k t+k+5 t}{2\left(k+1\right)}<2(k+2)t\leq n$, we have \begin{equation}\label{equ6}
	f_5(n)\geq f_5(2(k+2)t)=k^{3} t^{2}+\left(4 t^{2}-t\right) k^{2}+\left(2 t^{2}-t\right) k-4 t^{2}+3 t-1>0.
\end{equation}

By (\ref{equ3}), (\ref{equ4}), (\ref{equ5}), (\ref{equ6}) and $s\geq t+1$, we have $f_1(\rho_2)=(s-t)(k+1)f_3(\rho_2)> (s-t)(k+1)f_3((n-(k+1)t-1)=(s-t)(k+1)f_4(s)\geq (s-t)(k+1)f_4(\frac{n-1}{k+2})=\frac{(s-t)(k+1)}{k+2}f_5(n)\geq\frac{(s-t)(k+1)}{k+2}f_5(2(k+2)t)>0.$

Now we proceed to prove $f_1'(x)>0$ in the interval $[\rho_2,+\infty)$.

Since $f_1'(x)=3 x^{2}-2 \left(n-k s-s-2\right) x-k \,s^{2}+k s-s^{2}-n+s+1$ and $\frac{n-ks-s-2}{3}<n-(k+1)t-1<\rho_2$, then when $x\geq\rho_2$, $f_1'(x)\geq f_1'(\rho_2)>f_1'(n-(k+1)t_1)=-\left(k+1\right) s^{2}-\left(k+1\right) \left(2t k-2n+2t+1\right) s+3 k^{2} t^{2}+\left(-4 n t+6 t^{2}+2 t\right) k+n^{2}+\left(-4 t-1\right) n+3 t^{2}+2 t\triangleq f_6(s)$. By $\frac{-(k+1)(2kt-2n+2t+1)}{2(k+1)}=n-kt-t-\frac12>\frac{n-1}{k+2}\geq s\geq t+1$, we have $f_6(s)\geq f_6(t+1)=n^{2}+\left(\left(-2 t+2\right) k-2 t+1\right) n+\left(t^{2}-2 t\right) k^{2}+\left(t^{2}-5 t-2\right) k-3 t-2\triangleq f_7(n)$. Due to $\frac{\left(\left(2 t-2\right) k+2 t-1\right)}{2}=t+kt-k-\frac12<2(k+2)t\leq n$, we derive $f_7(n)\geq f_7(2(k+2)t)=\left(k^{2}+5 k+8\right) t^{2}+\left(2 k^{2}+5 k+1\right) t-2 k-2>0$. By the above arguments, $f_1'(x)>0$ for $x\geq \rho_2$, which implies $f_1(x)\geq f_1(\rho_2)>0$ when $x\geq \rho_2$. Hence $\rho_1=\rho(G^*)<\rho_2=\rho(H)<\rho(\widetilde{H})$, a contradiction with $\rho(G^*)\geq \rho(H)$ or $\rho(\widetilde{H})$.

{\noindent\textbf{Subcase 2.2.}} $n=(k+2)s$.

In this subcase, $G^*=K_s\vee (k+1)sK_1$ and $s=\frac{n}{k+2}\geq \frac{2(k+2)t}{k+2}=2t$. By Proposition \ref{leaf degree}, $\rho(G^*)$ is the largest root of $f_8(x)=0$, where $f_8(x)=x^{2}-\left(s-1\right) x-s^{2} \left(k+1\right)$. 

Recall that $\rho_2=\rho(H)$ is the largest root of $f_2(x)=0$. Thus we have \begin{equation}\label{equ7}
	\rho_2 f_8(\rho_2)=\rho_2 f_8(\rho_2)-f_2(\rho_2)=f_9(\rho_2),
\end{equation} where $f_9(x)=\left(\left(k+1\right) s-\left(k+1\right) t-1\right) x^{2}+(\left(k+2\right) s+t^{2}-t-k t+k \,t^{2}-k \,s^{2}-s^{2}-1) x+k^{2} t^{3}-k \,t^{2} \left(k+2\right) s+3 k \,t^{3}+k \,t^{2}-\left(k+2\right) s \,t^{2}+2 t^{3}+t^{2}$. Clearly, $\frac{\left(k+2\right) s+t^{2}-t-k t+k \,t^{2}-k \,s^{2}-s^{2}-1}{-2((k+1)s-(k+1)t+1)}<(k+2)s-(k+1)t-1$ by $s\geq 2t$, and then $f_9(x)$ is decreasing in the interval $[(k+2)s-(k+1)t-1,+\infty)$. By $\rho_2>n-(k+1)t-1=(k+2)s-(k+1)t-1$, we deduce \begin{equation}\label{equ8}\begin{aligned}
		f_9(\rho_2)>f_9((k+2)s-(k+1)t-1)=(k+1)f_{10}(s).\end{aligned}
\end{equation} where $f_{10}(x)=\left(k^{2}+3 k+2\right) s^{3}-\left(\left(3 k^{2}+9 k+7\right) t+2 k+3\right) s^{2}+(1+\left(3 k^{2}+8 k+5\right) t^{2}+\left(4 k+6\right) t) s-t (t k+1) \left(\left(k+2\right) t+1\right)$. Then we possess $f_{10}'(x)=3 (k^{2}+3 k+2) s^{2}-2 ((3 k^{2}+9 k+7) t+2 k+3)  s+1+(3 k^{2}+8 k+5) t^{2}+\left(4 k+6\right) t.$ If $x\geq 2t>\frac{(3 k^{2}+9 k+7) t+2 k+3}{3 (k^{2}+3 k+2)}$, then $f_{10}'(x)\geq f_{10}'(2t)=\left(3 k^{2}+8 k+1\right) t^{2}-\left(4 k+6\right)+1>0,$ which implies $f_{10}$ is increasing in the interval $[2t,+\infty)$.

Since $s\geq 2t$, we have $f_{10}(s)\geq f_{10}(2t)=t\left(\left(k^{2}+2 k-2\right) t^{2}+\left(-2 k-2\right) t+1\right)>0$ for $k\geq2$. For $k=1$, we admit $n=3s\geq 2(k+2)t=6t$ and then $n=6t$ or $n\geq 6t+3$. If $n\geq 6t+3$, then $s\geq 2t+1$ and thus $f_{10}(s)\geq f_{10}(2t+1)=t^{3}+8 t^{2}+8 t+2>0$. 

By (\ref{equ7}) and (\ref{equ8}), we admit $\rho_2 f_8(\rho_2)=f_9(\rho)>(k+1)f_{10}(s)>0,$ which yields $f_8(\rho_2)>0$. Clearly, $f_8(x)$ is increasing in the interval $[\rho_2, +\infty)$. Then $\rho(H)=\rho_2>\rho(G^*)\geq\rho(G)\geq\rho(H)$, a contradiction.

If $(k,n)=(1,6t)$ and $t\geq 3$, then $\rho(\widetilde{H})<\rho(H)$ by Proposition \ref{leaf degree}, which implies $\rho(\widetilde{H})=\rho(G^*)\geq\rho(G)\geq \rho(H)>\rho(\widetilde{H})$, a contradiction. If $(k,n)=(1,6t)$ and $t\in \{1,2\}$, then $\rho(\widetilde{H})=\rho(G^*)\geq\rho(G)\geq \rho(\widetilde{H})$, which implies $G=G^*\cong \widetilde{H}$.

Combining Case 1 and Case 2, the Theorem \ref{tb} is proved. {\hfill $\blacksquare$ \par}

\subsection{Proof of Theorem \ref{tc}} 
\hspace{1.5em}Let $H=K_t\vee(K_{n-(k+2)t}\cup (k+1)tK_1)$ and $G$ be a connected graph of order $n\geq 2(k+3)t+1$ with $\delta(G)\geq t\geq 1$ and $q(G)\geq q(H)$. Suppose to the contrary that $G$ contains no spanning tree with leaf degree at most $k$ and then it suffices to show that $G\cong H$ by Proposition \ref{leaf degree}. By applying Theorem \ref{t1}, there exists some nonempty subset $S\subseteq V(G)$ such that $i(G-S)\geq (k+1)|S|$. For convenience, we take $|S|=s$, then $n\geq (k+2)s$ and $s\geq \delta(G)\geq t$. Then $G^*\cong K_s\vee(K_{n-(k+2)s}\cup (k+1)sK_1)$ is the graph with the maximum size such that $G$ is a spanning subgraph of $G^*$ and $i(G^*-S)\geq(k+1)|S|$. Therefore, $q(G^*)\geq q(G)\geq q(H)$ by Lemma \ref{rq1}.

In what follows, we consider the following two cases to prove Theorem \ref{tc}.

{\noindent\textbf{Case 1.}} $s=t$.

In this case, $G^*=H$, then $q(G^*)=q(H)$ and thus $q(H)=q(G^*)\geq q(G)\geq q(H)$, which implies $G=G^*\cong H$.

{\noindent\textbf{Case 2.}} $s\geq t+1$.

Combining $n\geq (k+2)s$, it suffices to deal with the following two subcases.

{\noindent\textbf{Subcase 2.1.}} $n\geq(k+2)s+1$.

In terms of Proposition \ref{leaf degree}, $q_1=q(G^*)$ and $q_2=q(H)$ are the largest root of $g_1(x)=0$ and $g_2(x)=0$, respectively, where $g_1(x)=x^{3}-\left(3 n-2 k s-s-4\right) x^{2}-(2 k n s+4 k \,s^{2}-4 k s-2 n^{2}-n s+4 s^{2}+6 n-4) x-2 k^{2} s^{3}+4 k n \,s^{2}-4 k \,s^{3}-6 k \,s^{2}-2 n^{2} s+4 n \,s^{2}-2 s^{3}+6 n s-6 s^{2}-4 s$ and $g_2(x)=x^{3}-\left(3 n-2 k t-t-4\right) x^{2}-\left(2 k n t+4 k \,t^{2}-4 k t-2 n^{2}-n t+4 t^{2}+6 n-4\right) x-2 k^{2} t^{3}+4 k \,t^{2} n-4 k \,t^{3}-6 k \,t^{2}-2 t \,n^{2}+4 t^{2} n-2 t^{3}+6 n t-6 t^{2}-4 t$. Together with Lemma \ref{rq1} , we obtain $q=q(H)>q(K_{n-(k+1)t}\cup(kt+t)K_1)>2(n-(k+1)t-1)$, since $K_{n-(k+1)t}\cup(k+1)tK_1$ is a spanning subgraph of $H$. 

By direct calculation, we deduce \begin{equation}\label{equ12}
	\begin{aligned}g_1(q_2)=g_1(q_2)-g_2(q_2)=(s-t)g_3(q_2),\end{aligned}\end{equation} where $g_3(x)=\left(2 k+1\right) x^{2}-\left(\left(4 s+4 t+2 n-4\right) k+4 s+4 t-n\right) x-2 \left(k+1\right)^{2} s^{2}-2 (k+1)^{2} t^{2}-2 k^{2} s t+\left(\left(4 n-4 t-6\right) s+\left(4 n-6\right) t\right) k+\left(4 n-2 t-6\right) s+(4 n-6) t-2 (n-1) (n-2)$. 

By $n\geq\max\{(k+2)s+1,2(k+3)t+1\}$, we have $\frac{(4s+4t+2n-4)k+4s+4t-n}{2(2k+1)}<2(n-(k+1)t-1)$ and then $g_3(x)$ is increasing in the interval $[2(n-(k+1)t-1),+\infty)$, which implies \begin{equation}\label{equ13}
	\begin{aligned}g_3(q_2)>&g_3(2(n-(k+1)t-1))\\=&2(k+1)(-\left(k+1\right) s^{2}+\left(3 k t-2 n+3 t+1\right) s+4 k^{2} t^{2}-6 k n t+9 k \,t^{2}\\&+4 k t+2 n^{2}-7 t n+5 t^{2}-2 n+5 t).\end{aligned}\end{equation} 

Let $g_4(x)=-\left(k+1\right) x^{2}+\left(3 k t-2 n+3 t+1\right) x+4 k^{2} t^{2}-6 k n t+9 k \,t^{2}+4 k t+2 n^{2}-7 t n+5 t^{2}-2 n+5 t$. Since $\frac{(3 k t-2 n+3 t+1)}{2(k+1)}\leq t+1\leq s\leq \frac{n-1}{k+2}$, we have \begin{equation}\label{equ14}
	\begin{aligned}g_4(s)\geq g_4(\frac{n-1}{k+2})=\frac{1}{(k+2)^2}g_5(n),\end{aligned}
\end{equation} where $g_5(x)=\left(2 k^{2}+5 k+3\right) x^{2}-\left(6 k^{3} t+\left(28 t+2\right) k^{2}+\left(43 t+3\right) k+22 t\right) x+4 k^{4} t^{2}+\left(25 t^{2}+4 t\right) k^{3}+\left(57 t^{2}+18 t\right) k^{2}+\left(56 t^{2}+27 t-2\right) k+20 t^{2}+14 t-3$. Therefore, $g_5(x)$ is increasing in the interval $[2(k+3)t+1,+\infty)$ by $n\geq 2(k+3)t+1$ and so $
g_5(n)\geq g_5(2(k+3)t+1)=t((t+2) k^{3}+(7 t+16) k^{2}+(6 t+38) k-4 t+28)>0.$ 

By (\ref{equ12}), (\ref{equ13}), (\ref{equ14}) and $s\geq t+1$, we obtain $
g_1(q_2)=(s-t)g_3(q_2)>2(k+1)(s-t)g_4(s)\geq 2(k+1)(s-t)g_4(\frac{n-1}{k+2})\leq \frac{2(k+1)(s-t)}{(k+2)^2}g_5(n)=\frac{2(k+1)(s-t)}{(k+2)^2}g_5(2(k+3)t+1)>0.$

Now we show $g_1(x)$ is increasing in the interval $[q_2,+\infty)$. 

Clearly, $g_1'(x)=3 x^{2}-2 \left(3 n-2 k s-s-4\right) x-2 k n s-4 k \,s^{2}+4 k s+2 n^{2}+n s-4 s^{2}-6 n+4$. When $x\geq q_2$, we have $g_1'(x)\geq g_1'(q_2)>g_1'(2(n-(k+1)t-1))=-4\left( k+1\right) s^{2}+\left(-4\left(2k+1\right) \left(k+1\right) t+\left(6 n-4\right) k+5 n-4\right) s+12 \left(k+1\right)^{2} t^{2}-4 \left(3n-2\right) \left(k+1\right) t+2 n^{2}-2 n\triangleq g_6(s)$. Since $n\geq 2(k+3)t$ and $\frac{-4(2k+1)(k+1)t+(6n-4)k+5n-4}{8(k+1)}\geq \frac{n-1}{k+2}\geq s\geq t+1$, $g_6(s)\geq g_6(t+1)=2n^2+(6k+3-(6k+7)t)n+4(k+1)^2t^2-(8k^2+16k+8)t-8k-8\triangleq 2(2(k+3)t+1)^2+(6k+3-(6k+7)t)(2(k+3)t+1)+4(k+1)^2t^2-(8k^2+16k+8)t-8k-8=\left(6 k+34\right) t^{2}+\left(4 k^{2}+28 k+27\right) t-2 k-3>0$. Hence when $x\geq q_2$, $g_1'(x)\geq g_1'(q_2)>g_6(s)\geq g_6(t+1)>0$. 

Then we have $g_1(x)\geq g_1(q_2)>0$ when $x\geq q_2$, which implies $q(G^*)=q_1<q_2=q(H)$, a contradiction with $q(G^*)\geq q(H)$.

{\noindent\textbf{Subcase 2.2.}} $n=(k+2)s$.

In this subcase, $G^*=K_s\vee (k+1)sK_1$ and $s=\frac{n}{k+2}\geq \frac{2(k+3)t+1}{k+2}$. By Proposition \ref{leaf degree}, $q(G^*)$ is the largest root of $g_7(x)=0$, where $g_7(x)=x^{2}-\left(k s+4 s-2\right) x+2 s^{2}-2 s$. Recall that $q_2=q(H)$ is the largest root of $g_2(x)=0$ and $q_2>2(n-(k+1)t-1)=2((k+2)s-(k+1)t-1)$. Then we have \begin{equation}\label{equ19}\begin{aligned}
		q_2g_7(q_2)=&q_2g_7(q_2)-g_2(q_2)=g_8(q_2),
\end{aligned}\end{equation} where $g_8(x)=\left(2 k s-2 k t+2 s-t-2\right) x^{2}+(-2 k^{2} s^{2}+2 k^{2} s t-8 k \,s^{2}+3 t k s+4 k \,t^{2}+6 k s-4 k t-6 s^{2}-2 t s+4 t^{2}+10 s-4) x+2 k^{2} s^{2} t-4 t^{2} k^{2} s+2 k^{2} t^{3}+8 k \,s^{2} t-12 t^{2} k s+4 k \,t^{3}-6 t k s+6 k \,t^{2}+8 s^{2} t-8 s \,t^{2}+2 t^{3}-12 t s+6 t^{2}+4 t$. By $\frac{-2 k^{2} s^{2}+2 k^{2} s t-8 k \,s^{2}+3 t k s+4 k \,t^{2}+6 k s-4 k t-6 s^{2}-2 t s+4 t^{2}+10 s-4}{-2(2 k s-2 k t+2 s-t-2)}<2((k+2)s-(k+1)t-1)<q_2$, we have \begin{equation}\label{equ18}
	\begin{aligned}
		g_8(q_2)>&g_8(2((k+2)s-(k+1)t-1))\\=&4 (k+1) ((s-t) k+2 s-t-1)((k+1) s^{2}-(3 k t+4 t+1 s\\&+2 k \,t^{2}+\frac{5 t^{2}}{2}+2 t).
	\end{aligned}
\end{equation} 

Let $g_9(x)=\left(k+1\right)x^{2}-\left(3 k t+4 t+1\right) x+2 k \,t^{2}+\frac{5 t^{2}}{2}+2 t$. Clearly, $g_9(x)$ is increasing in the interval $[2t,+\infty)$ and then $g_9(s)\geq g_9(\frac{2(k+3)t+1}{k+2})=\frac{(t^{2}+2 t) k^{2}+(4 t^{2}+8 t) k-4 t^{2}-2}{2 (k+2)^{2}}>0$.

By (\ref{equ19}) and (\ref{equ18}), we have $q_2g_7(q_2)=g_8(q_2)>g_8(2((k+2)s-(k+1)t-1))=4 (k+1) ((s-t) k+2 s-t-1)g_9(s)>0$ and then $g_7(q_2)>0$. Due to $\frac{ks+4s-2}{2}<2((k+2)s-(k+1)t-1)<q_2$, $g_7(x)$ is increasing in the interval $[q_2,+\infty)$. Then when $x\geq q_2$, $g_7(x)\geq g_7(q_2)>0$ and thus $q(H)=q_2>q(G^*)\geq q(H)$, a contradiction.

Combining Case 1 and Case 2, the Theorem \ref{tc} is proved. {\hfill $\blacksquare$ \par}

\section*{Funding}
\hspace{1.5em}This work is supported by the National Natural Science Foundation of China (Grant Nos. 12371347, 12271337).

\section*{Declarations}

\textbf{Conflict of interest}\  The authors declare that they have no known competing financial interests or personal relationships that could have appeared to influence the work reported in this paper.

\vskip 0.5em

\noindent\textbf{Data availability} \  No data was used for the research described in the article.

\end{document}